\documentclass[11pt]{article}
\usepackage[margin=2.9cm]{geometry}
\RequirePackage[colorlinks,citecolor=blue,urlcolor=blue]{hyperref}

\usepackage{amssymb}
\usepackage{amsmath}
\usepackage{amsthm}
\usepackage{verbatim}
\usepackage{graphicx}
\usepackage{epstopdf}
\usepackage{amsmath}
\usepackage{bm}
\usepackage{setspace}
\usepackage[inline, shortlabels]{enumitem}
\usepackage{microtype}
\usepackage{enumitem}
\usepackage{subcaption}
\usepackage{booktabs}
\usepackage{float}
\usepackage{authblk}

\usepackage{multirow}

\usepackage{mathtools}
\usepackage{tikz}
\usetikzlibrary{decorations.pathmorphing}

\newcommand{\xrsquigarrow}[1]{\mathrel{\begin{tikzpicture}[baseline={(current bounding box.south)}]
        \node[inner sep=.5ex] (a) {$\scriptstyle #1$};
        \path[draw,<-,decorate,
              decoration={zigzag,amplitude=.7pt,segment length=1.2mm,pre=lineto,pre length=4pt}]
            (a.south east) -- (a.south west);
    \end{tikzpicture}}}

\usepackage{xparse}
\NewDocumentCommand{\permto}{ >{\SplitArgument{2}{,}}m }{\permtoaux#1}
\NewDocumentCommand{\permtoaux}{mmm}{#1 \xrightarrow{#2} #3}

\NewDocumentCommand{\apermto}{ >{\SplitArgument{2}{,}}m }{\apermtoaux#1}
\NewDocumentCommand{\apermtoaux}{mmm}{#1 \xrsquigarrow{#2} #3}

\usepackage{tikz-cd}
\usepackage[all,cmtip]{xy}

\usepackage[font=small]{caption}

\RequirePackage[colorlinks,citecolor=blue,urlcolor=blue]{hyperref}
\usepackage[citestyle=alphabetic,
			bibstyle=alphabetic,
			backend=bibtex,
			doi=false,
			isbn=false,
			url=false,
                maxnames=10,
			]{biblatex}

\bibliography{refs}
\renewbibmacro{in:}{}			
\DeclareFieldFormat[article,inbook,incollection,inproceedings,patent,thesis,unpublished]{citetitle}{#1}
\DeclareFieldFormat[article,inbook,incollection,inproceedings,patent,thesis,unpublished]{title}{#1}

\usepackage{multicol}
\usepackage{array, makecell}
\usepackage{algorithm}
\usepackage{algpseudocode}
\algrenewcommand\algorithmicrequire{\textbf{Input:}}
\algrenewcommand\algorithmicensure{\textbf{Output:}}

\makeatletter
\newcommand*\bigdot{\mathpalette\bigdot@{.5}}
\newcommand*\bigdot@[2]{\mathbin{\vcenter{\hbox{\scalebox{#2}{$\m@th#1\bullet$}}}}}
\makeatother

\newcommand\alli{*}
\usepackage{caption}
\usepackage{subcaption}
\usepackage[matrixnorms,opnorm=op]{arash_macros}
\newcommand{\nhamming}{\ensuremath{d_{\text{NH}}}}
\newcommand{\hamming}{\ensuremath{d_\text{H}}}
\newcommand\mis{\operatorname{Mis}}
\newcommand\epst{\widetilde \eps}
\newcommand\Dc{\mathcal D}
\newcommand\Gc{\mathcal G}

\newcommand\df{d}

\newcommand\convd{\Rightarrow}

\newcommand\mh{\widehat m}
\newcommand\hh{\widehat h}
\newcommand{\Dh}{\widehat D}
\newcommand{\Eh}{\widehat E}
\newcommand\Ac{\mathcal A}
\newcommand\Bc{\mathcal B}

\newcommand\pmin{\pi_{\min}}
\newcommand\sums{\mathcal S}

\newcommand{\ph}{\widehat p}
\newcommand\Ec{\mathcal E}
\newcommand\kb{\bar \kappa}
\newcommand\maxnorm[2][]{\norm[#1]{#2}_{\max}}

\newcommand\class{\mathcal C}

\newcommand\lipnorm[1]{\norm{#1}_{\operatorname{Lip}}}

\newtheorem{theorem}{Theorem}[section]
\newtheorem{lemma}{Lemma}[section]
\newtheorem{proposition}[theorem]{Proposition}
\newtheorem{remark}[theorem]{Remark}

\DeclareMathOperator\sbm{\mathsf{SBM}}
\DeclareMathOperator\rdpg{\mathsf{RDPG}}
\DeclareMathOperator\graphon{\mathsf{Graphon}}
\DeclareMathOperator\unif{\mathsf{U}}
\DeclareMathOperator{\bin}{\mathsf{Bin}}
\DeclareMathOperator\bern{\mathsf{Ber}}
\DeclareMathOperator\bernb{\overline{\mathsf{Ber}}}
\newcommand\given{\,|\,}

\newcommand\Bh{\widehat B}
\newcommand\perm{\Pi}
\newcommand\match{\mathcal M}
\newcommand\conn{\mathcal B}
\newcommand\counts{\mathcal C}
\newcommand\Ph{\widehat P}
\newcommand\Th{\widehat T}
\newcommand\Bt{\widetilde B}

\newcommand\zh{\widehat z}
\newcommand\nh{\widehat n}
\newcommand\Sh{\widehat S}
\newcommand\ind[1]{1\{#1\}}

\newcommand\mb{\bar m}
\newcommand\sigh{\widehat \sigma}

\DeclareMathOperator\lap{LAP}
\newcommand\Qh{\widehat Q}
\newcommand\Delh{\widehat \Delta}
\newcommand\St{\widetilde S}
\newcommand\Pt{\widetilde P}
\usepackage{tikz-cd}

\newcommand\allone{\bm 1}
 \newcommand\dmatch{d_F^{\operatorname{ma}}}

\theoremstyle{definition}
\newtheorem{prob}{Problem}

\title{Network two-sample test for block models} 
\author{Chung Kyong Nguen, Oscar Hernan Madrid Padilla and Arash A. Amini}
\affil{Department of Statistics, UCLA}

\begin{document}

\maketitle

\begin{abstract}
    We consider the two-sample testing problem for networks, where the goal is to determine whether two sets of networks originated from the same stochastic model. Assuming no vertex correspondence and allowing for different numbers of nodes, we address a fundamental network testing problem that goes beyond simple adjacency matrix comparisons. We adopt the stochastic block model (SBM) for network distributions, due to their interpretability and the potential to approximate more general models. The lack of meaningful node labels and vertex correspondence translate to a graph matching challenge when developing a test for SBMs.
    We introduce an efficient algorithm to match estimated network parameters,
    allowing us to properly combine and contrast information within and across samples, leading to a powerful test. We show that the matching algorithm, and the overall test are consistent, under mild conditions on the sparsity of the networks and the sample sizes, and derive a chi-squared asymptotic null distribution for the test.
     Through a mixture of theoretical insights and empirical validations, including experiments with both synthetic and real-world data, this study advances robust statistical inference for complex network data.
\end{abstract}

\section{Introduction}

Network data is pervasive across various fields, including transportation \cite{cardillo2013emergence}, trading \cite{barigozzi2010multinetwork}, social networks \cite{eagle2006reality,yu2021optimal}, neuroscience \cite{prevedel2014simultaneous,park2015anomaly,braun2021brain,padilla2022change}, ecology \cite{shen2022bayesian}, and politics \cite{amini2013pseudo}, among others. To address the diverse needs of these applications, recent statistical methods have emerged to handle scenarios where multiple networks are observed~ \cite{kolaczyk2019averages, mukherjee2017clustering,  josephs2023nested, amini2024hierarchical}. In this paper, we  consider the problem of testing whether two sets of networks have been generated from the same probability model. Specifically, in 
 network
 two-sample testing,  we are given two collections of graphs  $\{G_{1 1},\ldots, G_{1 N_1 }\}$ and  $\{G_{2 1 },\ldots, G_{2 N_2} \}$ generated as $G_{rt} \sim F_r$ for $r \in \{1,2\}$, $t \in [N_r] := \{1,\ldots,N_r\}$, and would like to test
\begin{equation}
\label{eqn:test}
H_0: F_1 = F_2 \quad{} \text{against} \quad H_1: F_1\neq F_2.
\end{equation} 
We approach the two-sample testing problem in (\ref{eqn:test}) without requiring the existence of vertex correspondence. This means that the nodes in different networks  are unlabeled, and as a consequence, the $i$th node in the $t$th network of the $r$th sample has no direct correspondence or meaning in any of the other networks. Additionally, it is possible for the number of nodes, denoted as $n_{rt}$, to vary across different networks. As such, the problem we are studying is truly a network testing problem, and not immediately reducible to testing the distributions of a collection of  adjacency matrices.

To further elaborate on the subtlety of graph versus matrix testing, 
consider the adjacency matrices $A_{rt} \in \{0,1\}^{n_{rt} \times n_{rt}}$ associated with each graph $G_{rt}$ for $r\in \{1,2\}$ and $t \in [N_r]$, by fixing a particular node labeling for each graph. Then, graph-level distribution $F_1$ induces distributions on the adjacency matrices $A_{1t}, t \in [N_1]$. However, these distributions are not directly comparable since the dimensions of $\{A_{1t}\}$ could be different. Even if the dimensions are the same (i.e., graphs $\{G_{1t}\}$ all have the same number of nodes), the distributions of $\{A_{1t}\}$ depend on the particular labelings chosen for each graph. Changing the labelings, will change the distributions of $\{A_{1t}\}$, although we want to treat all such distributions as the same. In other words, 
when testing the equality of distributions of the adjacency matrices, the formulation must account for equality up to relabeling of the nodes, and disregard the potential size mismatches, to truly test at the level of graph distributions.

To model the distribution of the adjacency matrices, we adopt the stochastic block model (SBM) framework introduced in \cite{holland1983stochastic} as a foundational structure. The SBM is one of the most commonly used models in the literature due to its simplicity and interpretability. It is effectively the equivalent of the ``histogram'' for network distributions~\cite{olhede2014network}, capable of approximating (piecewise) smooth graphons~\cite{airoldi2013stochastic, Gao2015RateOptimal}.
Given the prevalence of histograms, or binning, in traditional testing problems, SBMs serve as the natural starting point for studying two-sample testing for networks.

 In this work, we assume that the adjacency matrices are generated from SBM distributions. 
In particular, an adjacency matrix $A_{rt}\in \{0,1\}^{n_{rt} \times  n_{rt} } $ is generated as follows. First, we draw a vector of random \emph{community} labels $z_{rt} =  (z_{rt1 }, \ldots,z_{rtn } )^{\top} \in [K]^n$ with entries that are independent draws from a categorical distribution with parameter $\pi_r  =  (\pi_{r1},\ldots, \pi_{rK})$,  where $\pi_{rl}>0$ for all $l$ and  $\sum_{l=1}^K \pi_{rl} =1$. Here, $K$ is the number of communities. Then, given the labels $z_{rt}$, the entries of $A_{rt}$, below the diagonal, are independently drawn as
\[
   (A_{rt})_{ij} \given z_{rt} \,\sim \,\bern( (B_{r})_{z_{rti}\,z_{rtj} }  ),  \quad i < j,
\]where $B_{r}  \in [0,1]^{K \times K }$ is a symmetric matrix of probabilities. We assume the networks to be undirected, with no self-loops, hence $A_{rt}$ is extended to a symmetric matrix with zero diagonal entries. Throughout, we write $A_{rt} \sim \sbm_{n_{rt}}(B_r,\pi_r)$ for $A_{rt}$ generated as described above. We often suppress the dependence on $n_{rt}$ and simply write $A_{rt} \sim \sbm(B_r,\pi_r)$.

Given  $A_{rt} \sim \sbm(B_r,\pi_r)$, the network testing problem~\eqref{eqn:test} is equivalent to the following
\begin{equation}
\label{eqn:test2}
H_0: B_1 = P B_2 P^{\top} \;\text{for some }\, P \in \Pi_K     \quad{} \text{against} \quad H_1: B_1 \neq P B_2 P^{\top} \;\text{for all }\, P \in \Pi_K,
\end{equation} 
where $\Pi_K$ is the set of $K \times K$ permutation matrices. The inclusion of permutation matrices $P$ in~\eqref{eqn:test2} is related to the subtlety of testing unlabeled network models. The fact that there is no pre-defined node label or vertex correspondence across networks leads to the meaninglessness of community labels in SBM.
Consequently, demanding that $B_1 = B_2$ under the null hypothesis is not meaningful. Put another way, SBM parameters $(B_r,\pi_r)$ are only identifiable up a permutation; see~\eqref{eq:perm:equiv}.

To develop a test for~\eqref{eqn:test2}, it is essential to address 
the significant challenge posed by
the fact that matrices $B_r$ are only identifiable up to permutations.
A potential approach to constructing a test statistics is to estimate the community labels for each network, using any of the many existing approaches (for instance \cite{amini2013pseudo}). Subsequently, one can proceed to estimate $B_r$ by computing block averages of the adjacency matrices, based on these labels. However, the challenge lies in how to effectively combine, and compare, the estimates from different networks when vertex correspondence is not assumed. This becomes particularly daunting, as matching estimated labels across different networks 
boils down to searching over the space of  permutations $\Pi_K$, which has  $K!$ elements, a non-trivial task even for moderate values of $K$. It is worth noting that this matching challenge exists even if we observe a single network for each of the two samples, since one has to properly match the communities of the two networks to be able to compare the connectivity matrices $B_1$ and $B_2$.

\subsection{Summary of results}

In this paper, we tackle the aforementioned matching and computational challenges and make the following contributions:

\begin{enumerate}
    \item We study a computationally efficient algorithm that produces a permutation matrix $\Pt \in \Pi_K$  to match estimates $\Bh_r$ of $B_r$ for $r=1,2$. The algorithm has the property that, under some regularity conditions, if $\fnorm{ B_r - P_r \Bh_r P_r^{\top}}$ is small enough for some permutation matrix $P_r$, for $r=1,2$, and with $\fnorm{\cdot}$ the usual Frobenius norm, then it holds that
   \begin{equation}
       \label{eqn:matchin_upper}
       \fnorm{\Bh_2 - \Pt \Bh_1 \Pt^T} \leq 
           \sum_{r=1}^2 \fnorm{B_r - P_r\Bh_r  P_r^T}.
   \end{equation}
    Thus, we can match $\Bh_1$ with $\Bh_2$ using the permutation matrix $\Pt$ and the error is not worse than the sum of the estimation errors associated with $\Bh_1$ and $\Bh_2$.  Hence, when $\{\Bh_r\}$ are consistent estimators of $\{B_r\}$ (up to permutation), then  the matching $\Pt$ is consistent.    The main condition for the result in (\ref{eqn:matchin_upper}) is an $(\eps, \delta)$-friendly assumption, on $B_r$ matrices, as detailed in Definition~\ref{def:eigen}. In this condition, $\delta>0$ is the minimum gap between the eigenvalues of $B_r$ and $\eps>0$ is related to the corresponding eigenvectors.

    \item We show that when our proposed test statistic $\Th_n$ (Algorithm~\ref{alg:sbm-ts}) is applied  to two SBM samples, with sizes satisfying $n \,\le\, n_{rt} \,\lesssim \,  n $, and the same connectivity matrix (up to permutation) $B \in [0,1]^{K \times K}$ satisfying the sparsity condition $\frac{\nu_n}{n} \,\le\, B_{k\ell} 
 \,\lesssim\, \frac{ \nu_n}{n}$ for all  $k, \ell \in [K]$, then $ \Th_n$ converges in distribution to $\chi^2_{K(K+1)/2}$,  a chi-squared distribution  with $K(K+1)/2$ degrees of freedom. Thus, under the null distribution, we characterize the asymptotic distribution of our test statistic.  This holds under very general conditions. Specifically, letting  $\epst \in [0,1] $ be the largest missclassification error for estimating the labels $z_{rt}$ for  $t \in [N_r]$, $r \in \{1,2\}$, we require that $\sqrt{ N \nu_n n} \,\epst   \,=\, o_p(1)$ and $ N \nu_n n \to \infty$,  where $N =\min\{N_1,N_2\}$.

 \item Our theory also shows that  if the largest misclassification error  $\epst $ satisfies 
 \begin{equation}
     \label{eqn:signal_condition}
     \epst   +  \sqrt{\frac{ \log n}{ N \nu_n} } \,\lesssim \,  (\nu_n/n)^{-1}\min_{P \in \Pi_K}    \fnorm{B_2 - PB_1P^T}
 \end{equation}
 then, with probability approaching one, it holds that 
 \[
 \Th_n \gtrsim N n^2\cdot \min_{P \in \Pi_K}    \fnorm{B_2 - PB_1P^T}.
 \]
 Thus, provided that $B_1$ is different from $B_2$ even after potentially matching, then our test statistic will grow  to infinity, showing the consistency of the test,  as long as there is enough signal as stated in \eqref{eqn:signal_condition}.

\item We provide a simple, computationally efficient algorithm for the two-sample testing of SBMs. We demonstrate the performance of the test for both synthetic and real-world datasets, compared to the existing methods. The synthetic data include  SBMs, but also generative models beyond SBMs, including random dot product graphs and graphons, demonstrating the wide range of applicability of our test.

\end{enumerate}
\subsection{Related work}
\label{sec:related_work}

Two-sample testing for labeled
networks with known vertex correspondence has been well studied in the literature. 
In this setting, the vertex correspondence refers to the graphs being built on the same vertex set and 
the 1-1 mapping of vertices from one graph to another is given. Notable efforts in the labeled case include the following: 
\cite{Gin17}, where the authors develop a test based
on the geometric characterization of the space of the graph Laplacians. \cite{Ghoshdastidar2020} consider the problem from a minimax testing perspective, focusing on the theoretical characterization of the minimax separation with respect to the number of networks and the number of nodes. \cite{Levin} develop a test based on an omnibus embedding of multiple networks into a single space. 
\cite{chen20}  develop a spectral-based test statistic that has an asymptotic Tracy--Widom law under null. The results from  \cite{chen20} were improved by \cite{chen21} who constructed a test with milder technical conditions for the theoretical performance.

Existing literature on testing (\ref{eqn:test}) with unlabeled networks with no vertex correspondence is limited, but there have been notable efforts. For instance, \cite{kolaczyk2019averages} develop a geometric and statistical framework for making inferences about populations of unlabeled networks. They accomplish this by providing a geometric characterization of the space of unlabeled networks and introducing a Procrustean distance that explores the space of permutations $\Pi_n$  which has  $n!$ elements. It is worth noting that one major limitation of such approach is that optimizing the Procrustean distance over the space of $\Pi_n$ is a NP-hard problem. \cite{tang2017} introduce a two-sample test for random dot product graphs. This test statistic is based on estimating the maximum mean discrepancy between the latent positions, and it is proven to converge to the maximum mean discrepancy between the true latent positions. However, it is important to note that the statistics proposed in \cite{tang2017} do not involve any matching between latent positions, which can potentially impact the test's power, as we demonstrate in Section~\ref{sec:rdpg}. Additionally, the lack of matching mechanism does not allow for testing with samples including more than one network. \cite{Maugis_2020} develop a two-sample test, based on subgraph counts, and characterize the null distribution under a graphon model.
As they also point out, in general, an infinite number of subgraph counts, or more precisely homomorphism densities, is needed to distinguish two graphons. For the so-called \emph{finitely forcible} subclass~\cite[Section 16.7.1]{lovasz2012large}, which includes SBMs, a finite number is enough, but this number could still be large; see~\cite{gunderson2023graph} for recent developments in consistent estimation of SBMs from homomorphism densities. Given the computational complexity of estimating subgraph counts for large subgraphs, the approach becomes prohibitive for large networks.

A relevant line of research is clustering network-valued data.  Although algorithms for clustering are not primarily designed for testing, they can be adapted to our setting.
\cite{mukherjee2017clustering} describe
two algorithms for clustering.
Of interest to us is their NCLM algorithm for unlabeled networks, which is based on comparing network log moments. We describe the algorithm in detail in Section~\ref{sec:experiments} and compare to it.

As discussed earlier, the construction of our test relies on the subroutine that matches the estimated connectivity matrices of two networks. This type of problem has been studied extensively in the theoretical computer science literature and is known as weighted graph matching. It is formulated as an optimization problem:
\begin{equation}\label{eq:optimal:matching}
 \match^{\operatorname{opt}}(B_1 \to B_2): \quad P^* \in \argmin_{P \in \Pi_K} \fnorm{B_2 - PB_1 P^{\top}}
\end{equation}
where $B_1,B_2 \in \mathbb{R}^{K \times K}$ are input matrices.   The problem  can be rewritten as 
\begin{equation}
    P^* = 
    \argmax_{P \in \Pi_K}  \; \tr(P^{\top} B_2 P B_1),
\end{equation}
known as quadratic assignment problem (QAP) which is NP-hard. One of the standard techniques is to relax the constraint set of permutation matrices to a more tractable set, e.g. its convex hull \cite{lyz2016, vog2015, Aflalo2015} or the set of orthogonal matrices \cite{feizi}, and then ``round" the solution to the set of permutation matrices. 
\cite{ume} proposes to solve the weighted matching by solving a linear assignment problem on a certain matrix derived from 
eigenvectors of the adjacency matrices of the two graphs. \cite{fan20a} show exact recovery of the permutation
with high probability for the Gaussian Wigner model by constructing a weighted similarity matrix between all pairs of outer products of the eigenvectors of the adjacency matrices and then projecting it onto $\Pi_K$.  
\cite{Aflalo2015} solve the graph matching problem for the class of $(\eps, \delta)$-friendly graphs by relaxing the quadratic assignment problem to the non-negative simplex and then subsequently projecting onto $\Pi_K$ by solving a linear assignment problem (LAP).
\cite{feizi} study the unweighted graph matching problem and show mean-field optimality methods over the Erd\H{o}s--R\'{e}nyi
 graphs using transformations of the leading eigenvectors to align the structures of the adjacency matrices. The idea of our main matching algorithm is mentioned in the passing in \cite[Section 39.4]{spielman2019spectral} in the context of testing isomorphism of two graphs, and is attributed to~\cite{LeightonMiller79}. We present  other ideas for matching with the potential for further development in Section~\ref{sec:discuss}.

\subsection{Notation}

We often consider a one-to-one correspondence between permutations matrices $P \in \mathbb{R}^{K \times K}$ and permutations $\sigma$ on the set $[K] := \{1,\ldots,K\}$. The correspondence $\sigma \mapsto P_\sigma$ is defined by the following identity:
$(P_\sigma v)_i = v_{\sigma(i)}$
for all $v \in \reals^K$. It follows that $(P_\sigma^{T} v)_i = v_{\sigma^{-1}(i)}$. It is also helpful to note that $[P_{\sigma}]_{i \alli} = e_{\sigma(i)}^{\top}$ and $[P_{\sigma}^{\top}]_{\alli j} = [P_{\sigma}]_{j \alli}^{\top} = e_{\sigma(j)}$. This implies that if $B$ is a $K\times K$ matrix, $[P_\sigma B P_\sigma^{\top}]_{ij} = B_{\sigma(i),\sigma(j)}$. The linear assignment problem is often written as $\max_\sigma \sum_{i=1}^K B_{i, \sigma(i)}$. The cost function in this case is equivalent to $\tr(B P_\sigma^{\top})$. This follows by noting $[B P_\sigma^{\top}]_{ij} = B_{i \alli}\, e_{\sigma(j)} = B_{i, \sigma(j)}$.

Let  $P$ and $Q$   be matrices with associated permutations $\sigma$ and $\tau$ respectively. Notice that $\sigma \circ \tau$ is the permutation associated with $QP$, since $(QPv)_i   =  (Pv)_{\tau(i)} = v_{  \sigma(\tau(i)) )}$.

For $p \in [1, \infty)$ and a $n \times m$ matrix $A$, the $\ell_p \to \ell_p$ operator norm of a matrix $A = (a_{ij})$ is given by $\mnorm{A}_{p} := \sup_{x \neq 0} \norm{Ax}_p/\norm{x}_p$. In the special cases $p = 1, \infty$, we have
		$\mnorm{A}_{1}  = \max_j \sum_i |a_{ij}|$ and 
		$\mnorm{A}_{\infty} = \max_i \sum_j |a_{ij}|$.
We also write $\maxnorm{A} = \max_{i,j} |A_{ij}|$.

For label vectors $z, \zh \in [K]^n$, we
write $\hamming(z,\zh) = \sum_{i=1}^n 1\{z_i \neq \zh_i\}$ for their Hamming distance, $\nhamming(z,\zh) = \hamming(z,\zh)/n$ for their the normalized Hamming distance, and $\mis(z, \zh) = \min_{\sigma \in \Pi_K} \nhamming(z, \zh \circ \sigma)$ for the corresponding misclassification rate.

\section{Matching Methodology}

In this section, we begin by describing the challenges associated with the matching problem, specifically the task of aligning two matrices, $B_1$ and $B_2$, by relabeling the communities. Subsequently, we introduce our proposed matching algorithm, which is at the heart of our test construction.

\subsection{Matching challenge}\label{sec:matching:challange}

From a statistical perspective, the challenge in testing samples from $\sbm(B, \pi)$ is that the parameters $(B, \pi)$ are identifiable only up to permutations. That is, for any permutation matrix 
\begin{align}\label{eq:perm:equiv}
\sbm(B, \pi) = \sbm(P BP^{\top}, P \pi)
\end{align}
as distributions. To illustrate the challenge more clearly, consider observing the two-sample problem with only two adjacency matrices
\[
 A_1 \sim \sbm(B_1, \pi_1), \quad A_2 \sim \sbm(B_2, \pi_2).
\]
Assume that the null hypothesis holds, where  we can assume $B_2 = B_1$. To devise a test, a natural first step is to fit an $\sbm$ to $A_r$, using any consistent community detection algorithm, to obtain the labels, from which we can construct an estimate $\Bh_r = \conn(A_r)$ of $B_r$, for $r=1,2$. The operator $\conn$ is a shorthand for the procedure that produces such estimate, the details of which are discussed in Section~\ref{sec:test_construction}.
Even though $B_1 = B_2$, in general, we will have 
\[
\Bh_2 \approx P^* \Bh_1 P^{*\top},
\]
for some permuation matrix $P^* \in \perm_K$,
since in each case the  communities will be estimated in an unknown (arbitrary) order. 
There is no way a priori to guarantee the same order of estimated communities in the two cases. This is a manifestation of the permutation ambiguity in~\eqref{eq:perm:equiv}. 

An alternative way of formulating the issue is to assume that the population matrix $B_2$ carries the ambiguous permutation:
$
    B_2 = P^* B_1 P^{*\top}
$
while the estimated $\Bh_r$ are close to their respective population version, with identity permutation for the matching: $\Bh_r \approx B_r$ for $r=1,2$.  To simplify future discussions, let us introduce the following notation for \emph{exact matching}:
\begin{align}
\label{eqn:exact}
    \permto{B_1, P^*, B_2}\quad  \iff
    B_2 = P^*B_1P^{*T}.
\end{align}
The above discussion shows that two-sample testing for SBMs requires solving the following subproblem:
\begin{prob}
    \label{def:noisy:match}
 Assume that $\permto{B_1, P^*, B_2}$ for $B_1,B_2 \in [0,1]^{K \times K}$. The \emph{noisy graph isomorphism} (or graph matching) 
 problem is to recover a matching permutation $P^* \in \Pi_K$ between $B_1$ and $B_2$, using only noisy observations
 $\Bh_r \approx B_r$ for $r=1,2$.
\end{prob}

As alluded to in Section~\ref{sec:related_work}, this problem is in general hard. In particular, finding an optimal matching in Frobenius norm~\eqref{eq:optimal:matching} reduces
to solving a quadratic assignment problem (QAP), which is NP-hard in general.

\subsection{Spectral matching}
For a large class of weighted networks, we can avoid solving a QAP to recover matching. The idea is the following: If the matrix $B_1$ (and hence $B_2$) has distinct eigenvalues, then an eigenvalue decomposition (EVD) contains the information needed to recover the permutation. To be able to solve the noisy matching problem, we need to assume slightly more:
\begin{defn}
\label{def:eigen}
Let $B \in [0, 1]^{K \times K}$ be the weighted adjacency matrix of an undirected graph on nodes $[K]$, and let $B = \sum_{k=1}^K \lambda_k q_k q_k^{\top}$ be the corresponding EVD. Then, $B$ is called $(\theta,\eta)$-friendly if 
\begin{align}
    \min_{1\le \,k\, \neq \ell \,\le\, K} |\lambda_k - \lambda_\ell| &> \eta,  \\
    |q_k^{\top} 1| &> \theta, \; k \in [K].
\end{align}
We say that $B$ is friendly if it is $(\theta,\eta)$-friendly for some $\theta, \eta > 0$.
\end{defn}
\begin{rem}
    In contrast to \cite{Aflalo2015}, our definition of $(\theta,\eta)$-friendly---which is often termed $(\eps,\delta)$-friendly in the literature---only requires a lower bound on $|q_k^{\top} 1|$ but no upper bound. Also, in the literature, $(\eps,\delta)$-friendly property is often stated as a property of graphs, whereas here we state the concept for 
    weighted adjacency matrices. 
\end{rem}

Our proposed spectral matching algorithm for friendly weighted graphs is outlined in Algorithm~\ref{alg:spectral_matching}. LAP$(\cdot)$ in Line~5 stands for a solution of the \emph{linear assignment problem}:
For a $K \times K$ matrix $Q$, 
\begin{equation}\label{eq:lap:def}
    \lap(Q) := \argmax_{P \in \Pi_K}\; \tr(PQ)
\end{equation}
where $\Pi_K$ is the set of $K \times K$ permutation matrices. It is well-known that LAP can be solved as a linear program. There are also efficient specialized algorithms for solving it, e.g. the Hungarian algorithm \cite{Kuhn1955Hungarian}. 

\begin{algorithm}[t]
\caption{Spectral Matching: $\match(\Bh_1 \to \Bh_2)$}
\label{alg:spectral_matching}
\begin{algorithmic}[1]
    \Require{$\Bh_1$, $\Bh_2$}
    \For{$r = 1, 2$}
        \State Perform EVD on $\Bh_r$ to obtain $\Bh_r = \Qh_r \Lambda_r \Qh_r^{\top}$.
    \EndFor
    \State Recover the diagonal signs 
     $\Sh_{ii} = \sign \bigl( [\Qh_2^{\top} \allone]_i / [\Qh_1^{\top} \allone]_i \bigr)$ and set $\Sh = \diag(\Sh_{ii})$.
    \State Recover the permutation matrix as $\Pt = \lap(\Qh_1 \Sh \Qh_2^{\top})$.
    \State \Return $\Sh$, $\Pt$.
\end{algorithmic}
\end{algorithm}

The rationale behind Algorithm~\ref{alg:spectral_matching} follows from the following lemma:
\begin{lem}\label{lemma:q2 = pq1s}
Consider two $K \times K$  matrices $B_1$ and $B_2$ with EVDs given by 
$B_r = Q_r \Lambda Q_r^{\top},\,r =1,2$,
for some diagonal matrix $\Lambda$ with distinct diagonal entries. Then, a permutation matrix $P^*$ satisfies \begin{align}\label{eq:B1:B2:Ps}
    \permto{B_1, P^*, B_2}
\end{align}
if and only if there exists a diagonal sign matrix $S^*$, such that
\begin{align}\label{eq:Q1:Q2}
    Q_2 = P^* Q_1 S^*.
\end{align}
Moreover if $B_1$ is $(\theta, \eta)$ friendly, then there is at most one $P^*$ satisfying~\eqref{eq:B1:B2:Ps}.
 \end{lem}

To understand the idea, let us assume that we apply Algorithm~\ref{alg:spectral_matching} to matrices $B_1$ and $B_2$ that have an exact matching as in~\eqref{eq:B1:B2:Ps}. Then, according to Lemma~\ref{lemma:q2 = pq1s}, 
\[
Q_2^\top \allone = S^* Q_1^\top P^{*\top} \allone = S^* Q_1^\top \allone.
\]
Recalling that $S^* = \diag(S^*_{ii})$, we obtain $S^*_{ii} = [Q_2^\top \allone]_i/[Q_1^\top \allone]_i$. For robustness in the noisy case, we can also take the sign of the right-hand side since $S^*_{ii}$ is known to be in $\{-1,1\}$. Once we recover the sign matrix $S^*$, and noting that $S^*$ is its own inverse,  we have from~\eqref{eq:Q1:Q2} that $P^* = Q_2 S^* Q_1^\top$. However, this equality only gives a valid permutation matrix in the noiseless case. In the noisy case, we can instead solve the optimization problem suggested by~\eqref{eq:Q1:Q2}:
\begin{align*}
    \argmin_{P  \in \Pi_K } \fnorm{Q_2 - P Q_1 S^*}^2 = \argmax_{P \in \Pi_K  } \tr(Q_2^\top P Q_1 S^*) = \lap(Q_1 S^* Q_2^\top)
\end{align*}
where the first equality is by $\fnorm{Q_2}^2 = \fnorm{P Q_1 S^*}^2 = K$ and second equality by the circular invariance of the trace. The final LAP problem is exactly what Algorithm~\ref{alg:spectral_matching} solves. 

 In the sequel, we write $\match(\Bh_1 \to \Bh_2)$ to denote a matching done by Algorithm~\ref{alg:spectral_matching} for input matrices $\Bh_1$ and $\Bh_2$. Although we focus on this algorithm for the most part, there are other efficient approaches inspired by Lemma~\ref{lemma:q2 = pq1s} which we briefly discuss in Section~\ref{sec:discuss}.

\section{Test construction}
\label{sec:test_construction}

We are now ready to describe our main algorithm for constructing an SBM two-sample test. First, let us briefly describe how the operator $\conn$, alluded to in Section~\ref{sec:matching:challange}, is implemented. Given an $n \times n$ adjacency matrix $A$, we apply a community detection algorithm (e.g. regularized spectral clustering~\cite{amini2013pseudo} or Bayesian community detection \cite{shen2022bayesian}) to obtain label vector $\zh = (\zh_1,\dots,\zh_n) \in [K]^n$.
Let $n_k = \sum_{i=1}^n \ind{\zh_i = k}$ be the number of nodes in the estimated community $k$. 
Define the block-sum operator $\sums : \reals^{n \times n} \times [K]^n \to \reals^{K \times K}$ and block-count operator $\counts: [K]^n \to \reals^{K \times K}$ as
\begin{align}
     [\sums(A, z)]_{k \ell} &= \sum_{i,j} A_{ij} \ind{z_i = k, z_j = \ell}, \\
    [\counts(z)]_{k\ell} &= n_k(z) \bigl(n_\ell(z)  - 1\{k = \ell\} \bigr).
\end{align}
where $n_k(z) = \sum_{i=1}^n 1\{z_i = k\}$.
Then, we have
\begin{align}\label{eq:conn:operator}
    \conn(A, z) = 
    \sums(A, z)  \oslash \counts(z),
\end{align}
where $\oslash$ denotes the Hadamard (i.e., elementwise) division of matrices. For diagonal blocks ($k = \ell$) a double-counting occurs in both the numerator and denominator of~\eqref{eq:conn:operator} which can be avoided, for computational efficiency, by restricting the sums to $i < j$.

\subsection{Main algorithm}
Algorithm~\ref{alg:sbm-ts} summarizes our approach to SBM two-sample testing. To simplify the discussion, we refer to permutation matrices that match two networks, simply as ``matching''.
We can divide the procedure into three stages: 
\begin{enumerate}
    \item First, we match each network in a given sample $r \in \{1,2\}$ to the first network in the sample (steps~\ref{step:randomize}--\ref{step:match:to:first}). 
    
    \item Having the (individual) matching matrices  $\Ph_{rt}$, we can align the estimated $\Bh_{rt}$ for each network to compute a single estimate $\Bh_r$ for each sample $r$. Then, we can perform a global matching of the two samples (steps~\ref{step:Bh:r}--\ref{step:global:matching}).
    
    \item  Armed with the global matching $\Ph$, in the third stage (steps~\ref{step:Shr:mhr}--\ref{step:Th}), we align all the block sums and counts across all the samples to form stronger estimates of $B_1$ and $B_2$, as $\Bh^{(1)} := \Sh_1  \oslash \mh_1$ and $\Bh^{(2)} := \Sh_2' \oslash \mh_2'$. These two estimates are also properly aligned, hence can be compared elementwise, leading to the formation of test statistic $\Th_n$ in~\eqref{eq:Th:def}. We also compute a global (i.e. pooled) estimate of the variance ($\sigh_{k\ell}^2$) for the $(k,\ell)$th entry of the difference $\Bh^{(1)} - \Bh^{(2)}$ in step~\ref{step:variance}. This allows use to properly normalize each entry of the difference  before aggregating  them to form $\Th_n$.

\end{enumerate}

\begin{algorithm}[t]
\caption{SBM Two-Sample Test (SBM-TS)}
\begin{algorithmic}[1] \setstretch{1.25} 
    \Require{Adjacency matrices $A_{rt}$ and initial label estimates $\zh_{rt}^{(0)}$, for $t \in [N_r]$ and $r = 1,2$.}
    \Ensure{Two-sample test statistic $\Th_n$.}
    \Statex \textit{Match to the first network within each sample $r$: }
    \For{$t = 1,2,\dots,N_r$ and $r = 1,2$}
        \State Set $\zh_{rt}= \sigma_{rt} \circ \zh_{rt}^{(0)}$ for independent random permutation $\sigma_{rt}$. \label{step:randomize}
        \State Set $\Sh_{rt} = \sums(A_{rt}, \zh_{rt})$ and $\mh_{rt} = \counts(\zh_{rt})$.
        \State Set $\Bh_{rt} = \Sh_{rt} \oslash \mh_{rt}$.  \label{step:bhat_rt}
        \State Set $\Ph_{rt} = \match( \Bh_{rt} \to \Bh_{r1})$.
        \label{step:match:to:first}
    \EndFor
    \Statex \textit{Find the global matching permutation:}
    \State Set $\Bh_r = \frac1{N_r} \sum_{t=1}^{N_r} \Ph_{rt} \Bh_{rt} \Ph_{rt}^{\top}$ for $r = 1,2$.
    \label{step:Bh:r}
    \State Set $\Ph =\match(\Bh_{2} \to \Bh_{1})$.
    \label{step:global:matching}
    
    \Statex \textit{Align the sums and counts, across all samples, and form aggregate estimates:}
    \State Set $\Sh_r = \sum_{t=1}^{N_r} \Ph_{rt} \Sh_{rt} \Ph_{rt}^{\top}$ and $\mh_r = \sum_{t=1}^{N_r} \Ph_{rt} \, \mh_{rt}\, \Ph_{rt}^{\top}$ for $r=1,2$.
    \label{step:Shr:mhr}
    \State Set $\Sh_2' = \Ph \Sh_2 \Ph^{\top}$ and $\mh'_2 = \Ph \mh_2 \Ph^{\top}$.
    \State Let $\Bh = (\Bh_{k\ell})$ where 
    $
    \Bh_{k\ell} = \dfrac{
        \Sh_{1 k\ell} \,+\, \Sh'_{2 k\ell}
    }{
        \mh_{1 k\ell} \,+\, \mh'_{2 k \ell}
    }
    $
    and set $\sigh_{k\ell}^2  = \Bh_{k\ell}(1-\Bh_{k\ell}).$
    \label{step:variance}

    \State Let $\Bh^{(1)} =\Sh_1  \oslash \mh_1$ and $\Bh^{(2)} := \Sh_2' \oslash \mh_2'$. \hfill (\textit{Stronger estimates})
    \label{step:Bhr:final}
    
    \State Let $\hh_{k\ell}$ be the harmonic mean of $\mh_{1k\ell}$ and $\mh'_{2 k \ell}$ and form the test statistic:
    \vspace{-2ex}\begin{align}\label{eq:Th:def}
     \Th_n := \sum_{k \le \ell} \frac{\hh_{k\ell}}{2\sigh^2_{k\ell}}  \bigl( 
     \Bh^{(1)}_{k\ell} -\Bh^{(2)}_{k\ell} \bigr)^2.
    \end{align}
    \vspace{-2ex}
    \label{step:Th} 
\end{algorithmic}
\label{alg:sbm-ts}
\end{algorithm}

To get an intuition for the final statistic, let us assume that all the matchings are correctly recovered and set them to identity. Let us also assume that the community detection algorithm has perfectly recovered the labels for each network, so that $\mh_r$ is equal to $m_r := \sum_{t=1}^{N_r} \counts(z_{rt})$ the ``true''  aggregate block count matrix. Assume that we are under null and let $B = B_1 = B_2$.  

Under the above  assumptions, $\Sh_{1k\ell} \sim \bin(m_{1k\ell}, B_{k\ell})$ and hence $\Bh^{(1)}_{k\ell} = \Sh_{1k\ell}/m_{1k\ell}$ is approximately distributed as $N(0, \sigma_{k\ell}^2/m_{1k\ell})$ where $\sigma_{k\ell}^2 = B_{k\ell}(1-B_{k\ell})$. A similar statement holds for  $\Bh^{(2)}$. Then,
\[
\sqrt{{h_{k\ell}}/{(2\sigma_{k\ell}^2)}}(\Bh^{(1)}_{k\ell} -\Bh^{(2)}_{k\ell}) \;\stackrel{d}{\approx}\; N(0,1)
\]
where $h_{k\ell}$ is the harmonic mean of $m_{1k\ell}$ and $m_{2 k\ell}$. These all suggest that if all the stars align, $\Th_n$ will be approximately distributed as $\chi^2_{K(K+1)/2}$. In practice, the matchings are nontrivial and the block counts and sums are based on imperfect labels.  Nevertheless, in Theorem~\ref{thm:null:main} we formalize the above intuition and show that $\Th_n$ indeed has the expected asymptotic distribution, under mild conditions.

\medskip

It is worth noting that our final estimates of $B_r$ are $\Bh^{(r)}$ in step~\ref{step:Bhr:final} rather than $\Bh_r$ in step~\ref{step:Bh:r}. Even if we ignore the matching issue, $\Bh^{(r)}$ are better estimates. To see this, consider i.i.d. draws $X, Y\sim \bin(n, p)$. Then, both $(X+Y) / (n+m)$ and $\frac12((X/n) + (Y/m))$ are unbiased for $p$, but the former has lower variance in general.
The randomization in step~\ref{step:randomize} is a technical device, ensuring that the matching permutations in step~\ref{step:match:to:first} are independent of the data $\{A_{rt}\}$, irrespective of the community detection algorithm used to produce the labels; see Lemma~\ref{lem:randomization}.

\section{Theory}

We start by showing that our proposed matching Algorithm~\ref{alg:spectral_matching} consistently recovers the correct permutations. We then derive the asymptotic null distribution of the statistic, followed by a result showing the consistency of the test under sparse alternatives.

\subsection{Matching consistency}
\label{sec:match:consist}

We consider connectivity matrices $B_1$ and $B_2$ that are $(\theta,\eta)$-friendly. Recall our shorthand notation for an exact matching, introduced in~\eqref{eq:B1:B2:Ps}. Similarly, it is helpful to introduce the following graphical mnemonic for approximate matching:
\begin{align}
\label{eqn:approximate}
\apermto{B_1, P, B_2} \quad \iff \quad 
   \fnorm{B_2 - P B_1 P ^{\top}}
    \leq \frac{\theta \eta}{2\sqrt{2}K}.
\end{align}

Then, we have following guarantee for the matching algorithm:

\begin{thm}[Matching consistency]\label{prop:matching:recovery}
Suppose that $B_1$ and $B_2$ are $K\times K$ connectivity matrices that are $(\theta, \eta)$-friendly with $\theta < \frac{1}{4\sqrt{K}}$, and with EVDs given by 
$B_r = Q_r \Lambda Q_r^{\top},\,r =1,2$. Moreover, assume 
$\permto{B_1, P^*, B_2}$
for some permutation matrix $P^*$. For $r = 1,2$, let $\Bh_r$ be an estimate of $B_r$ satisfying 
\begin{equation}\label{eq:eps:delta:assump}
    \apermto{\Bh_r, P_r, B_r}
\end{equation}
for some permutation matrices $P_r$. Let $\Qh_1,  \Qh_2$, $\Sh$ and $\Pt$ be as defined in Algorithm~\ref{alg:spectral_matching}, and let 

\begin{equation}
    S_r = \argmin_{S \in  \Psi_K } \fnorm{Q_r - P_r \Qh_r S}, \quad r=1,2,
\end{equation}
where $\Psi_K$ is the set of $K \times K $ diagonal sign matrices. Then, the following holds:
\begin{enumerate}[(a)]
\item \emph{Sign recovery:} $\Sh = \St := S_1 S^* S_2$ with $S^*$ as in Lemma \ref{lemma:q2 = pq1s}.

    \item\emph{Permutation recovery:} $\lap(\Qh_1 \St \Qh_2^{\top}) = \{\Pt\}$ where $\Pt := P_2^{\top} P^* P_1$, and 
	\begin{equation}
	   \fnorm{\Bh_2 - \Pt \Bh_1 \Pt^{\top}} \leq 
           \sum_{r=1}^2 \fnorm{B_r - P_r\Bh_r  P_r^{\top}}. 
	\end{equation}
\end{enumerate}
\end{thm}

The theorem is stated with general permutation matrices $P_r$ which is helpful in subsequent proofs. If we assume $P_r = I$ for simplicity, the result suggests that once $\Bh_r$ is close enough to $B_r$, then $\Pt = P^*$, that is Algorithm~\ref{alg:spectral_matching} recovers the correct matching exactly.

\begin{rem}
    One can also ask whether our matching algorithm is \emph{self-consistent}? That is, if we reverse the order of the $\Bh_1$ and $\Bh_2$, do we get permutations that are transposes of each other? The answer is yes. Notice that if we reverse the order of $\Bh_1$ and $\Bh_2$, the sign matrix $\hat{S}$ does not change. Moreover, 
    \[
    \permto{B_1, P^*, B_2}  \quad \text{implies}\quad    \permto{B_2, (P^*)^{\top}, B_1}.
    \]Hence,   $\lap(\Qh_2 \St \Qh_1^{\top}) = \{P^{\prime}\}$     where  $P^{\prime}   =    P_1^{\top} (P^*)^{\top} P_2 =  (P_2^{T} P^*  P_1 )^{\top} =  \Pt^{\top} $.  
\end{rem}

\begin{figure}[t]
	\centering
	\vspace{1.5cm}
	\begin{tikzcd}[column sep=large, row sep=large,transform canvas={scale=1.25}]
		\Bh_1 \arrow[dashed]{r}{P_1} 
            \arrow[bend left=20]{d}[swap]{\Pt} 
                & B_1 \arrow{d}{P^*} \\
		\Bh_2 \arrow[dashed]{r}[swap]{P_2} & B_2
	\end{tikzcd}
	\vspace{1.5cm}
	\caption{\label{fig:schematic}Schematic diagram of permutation recovery in Theorem~\ref{prop:matching:recovery}. The solid and dashed straight arrows correspond to exact and approximate match. The bent arrow represents an application of the matching algorithm $\match$.}  
\end{figure}
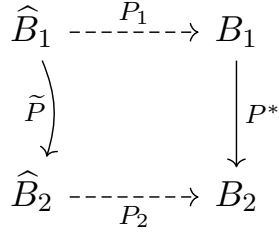

To gain a deeper understanding of why our target permutation matrix is $P_2^{\top} P^* P_1$, let us analyze Figure \ref{fig:schematic}. Within this illustration, each edge symbolizes a matching between matrices based on both the direction of the edge and the permutation matrix associated with it. To clarify the nature of these matchings, we use different types of arrows:
\begin{enumerate}
    \item Dashed straight arrows represent approximate matching. For instance, the top edge corresponds to $\apermto{\Bh_1, P_1, B_1}$, following the notation from Equation (\ref{eqn:approximate}).

\item In contrast, the solid straight solid arrow signifies exact matching corresponding to $\permto{B_1, P^*, B_2}$, with the notation from (\ref{eqn:exact}).
\end{enumerate}

Considering that the inverse of a permutation matrix is its transpose, we can trace the path from $\Bh_1$ to $\Bh_2$ in Figure \ref{fig:schematic}. We start from $\Bh_1$, progress through $B_1$, then $B_2$, and finally arrive at $\Bh_2$. The permutation matrices associated with this path are $P_1$, $P^*$ and $P_2^{\top}$ (the inverse of $P_2$). By multiplying these matrices together, we obtain the desired permutation matrix, which is $P_2^{\top} P^* P_1$.

\subsection{Null distribution}
To analyze the statistical properties of the test, we make the following ``sparsity'' and ``size'' assumptions
\begin{align}
 \frac{\nu_n}{n} &\,\le\, B_{k\ell} 
 \,\le\, \frac{C_1 \nu_n}{n}, \label{eq:sparsity} \\
   n &\,\le\, n_{rt} \le C_2 n \label{eq:size}
\end{align}
for all $k, \ell \in [K]$ and $t \in [N_r], r = 1,2$.
We also write $\pi_k = \pr(z_{rti} = k)$, recalling that we are under the null,
and set $\pmin = \min_k \pi_k$.

Let us fix some $\kappa \in (0,1)$ and $\alpha > 0$ and let $\beta = 1/(\kb \pmin)$ where $\kb = 1-\kappa$. 
For any $N$, let 
\begin{align}
     \delta_n^{(N)} := \sqrt{\frac{3 \beta^2 \alpha \log n}{ N n \nu_n}}.
\end{align}
For convenience, we assume that ($n$ is large enough so that)
\begin{align}\label{assu:simple:scaling}
    \delta_n^{(1)} \le 1, \quad 
\frac{\log n}{n} \le \frac{\kappa^2 \pmin}{3 \alpha}.
\end{align}
Moreover, define
\[
\gamma_n = C_1 \Bigl( 56 \beta^3 \cdot \epst + \delta_n^{(1)} \Bigr), \quad\text{where}\quad \epst := \max_{\substack{t \, \in\,  [N_r]\\  r\, =\, 1,2}} \,\mis(z_{rt}, \zh_{rt}).
\]
Suppose that $B$ is $(\theta, \eta)$-friendly with $\theta < \frac1{4 \sqrt{K}}$
and consider the condition
\begin{align}\label{assu:gamma}
\frac{\nu_n}n  \gamma_n \le \frac{\theta \eta}{2 \sqrt 2 K}. 
\end{align}
Since $B$ satisfies~\eqref{eq:sparsity}, in general $\eta \asymp \nu_n/ n$ but since eigenvectors are normalized, we can expect $\theta \asymp 1$. These scalings hold, for example, when $B = (\nu_n / n) B^0$ for fixed matrix $B^0$ which is $(\theta, \eta_0)$-friendly. Then, as long as $\gamma_n = o_p(1)$, we expect~\eqref{assu:gamma} to hold.

The following result guarantees an asymptotic null distribution for our proposed test statistic:

\begin{thm}[Null distribution]\label{thm:null:main}
    Assume that Algorithm \ref{alg:sbm-ts}
    is applied to two SBM samples, with sizes satisfying~\eqref{eq:size}, and the same connectivity matrix $B$ satisfying~\eqref{eq:sparsity} and $\maxnorm{B} \le 0.99$. Moreover, assume that $B$ is $(\theta,\eta)$-friendly with $\eta = \Omega(\nu_n/n)$
    and $\theta = \Omega(1)$.
    Let $N = \max\{N_1,N_2\}$ and assume that for some $\kappa \in (0,1)$ and $\alpha > 0$ and $c > 0$, we have
    \[
    \gamma_n = o_p(1), \quad \sqrt{ N \nu_n n} \,\epst   \,=\, o_p(1), \quad 
    N \nu_n n \to \infty,\quad 
    N  = o(n^\alpha),
    \]
    and $\min\{N_1, N_2\} \ge c N$.
    Then, for $\Th_n$, the output of Algorithm~\ref{alg:sbm-ts}, we have 
    \[
    \Th_n \convd \chi^2_{K(K+1)/2}.
    \]
\end{thm}

\begin{remark}
Notice that the conclusion in Theorem \ref{thm:null:main} holds under very mild assumptions. The condition on the misclassification error  is $\epst = o_{p}( (N \nu_n n)^{-1/2}  )$.  To see why this is a very weak requirement, recall Theorem 3.2 from
\cite{ zhang2016minimax} which translated to our notation shows that 
one can achieve misclassification rate as good as
\begin{align}\label{eq:optimal:miss:rate}
\mathbb{E}( \mis(z_{rt}, \zh_{rt}))\,\leq\,  \exp(- c\nu_n(1+ o(1) ) )    
\end{align}
for all $r =1,2$  $t =1,\,\ldots, N_r$, and for a constant $c>0$.  Here, $\zh_{rt}$ is a penalized likelihood estimator. Hence, by Markov’s inequality and union bound we obtain that
\[ 
 \mathbb{P}( \epst \geq \exp(- c\nu_n(1+ o(1) )/2   ) \,\leq\, \exp(- c\nu_n(1+ o(1) ) /2 + \log N_1 + \log N_2 ).
\]
Thus, for $\nu_n$ satisfying  $\nu_n \geq  \max\{\log N_1, \log N_2\}$, we obtain that  $\epst \,=\, o_p(\exp(- c\nu_n(1+ o(1) )/2   )$.  This is in general a much stronger condition than $\epst = o_{p}( (N \nu_n n)^{-1/2}  )$.
\end{remark}

\subsection{Test Consistency}

Next, we show that the test is consistent, in the sense that $\Th_n \to \infty$ under the alternative hypothesis that the two SBMs are different.  For two $K\times K$ matrices $B_1, B_2$, consider the pseudometric
\begin{equation}
    \dmatch(B_1, B_2) := \min_{P \in \Pi_K}
    \fnorm{B_2 - PB_1P^{\top}},
\end{equation}
where $\Pi_K$ is the set of $K \times K$ permutation matrices.

\begin{thm}[Consistency]\label{thm:test:consist}
    Assume that Algorithm~\ref{alg:sbm-ts} is  applied to two SBM samples, with sizes satisfying~\eqref{eq:size}, and connectivity matrices $B_1$ and $B_2$  satisfying~\eqref{eq:sparsity}. For $r=1,2$, let
    \[
    \xi_r := C_1 \bigl( 40 \beta^3 \epst + \delta_n^{(N_r)}   \bigr)
    \]
    and assume that
    \begin{align}\label{eq:dmatch:lower}
    \sqrt{12} K  \max_{r=1,2} \xi_r \le \frac{\dmatch(B_1,B_2)}{ \nu_n/n}.
    \end{align}
    Moreover, assume that~\eqref{assu:simple:scaling} holds and let $\Gc$ be the event that~\eqref{assu:gamma} and $48 C_2 \beta^4 \epst \le 1$ hold. Here, $C_1$ and $C_2$ are the same constants as in (\ref{eq:sparsity}).     Then, with probability at least $1 - 19 N_+ K^2 n^{-\alpha} - \pr(\Gc^c)$,
    \[
    \Th_n \ge \frac{N n^2}{12 \beta^2} \cdot \dmatch(B_1, B_2)^2.
    \]
\end{thm}

Theorem \ref{thm:test:consist} establishes that under conditions similar to those in the null case studied in Theorem \ref{thm:null:main}, provided there exists ample signal to differentiate between the two probability matrices (\ref{eq:dmatch:lower}), our proposed test exhibits consistency. Specifically, in the dense case ($\nu_n = n$), under the alternative $\dmatch(B_1,B_2)$ being a constant $> 0$, we have $\Th_n = \Omega(N n^2)$. In the sparse case, we can assume $\liminf \dmatch(B_1,B_2) / (\nu_n / n) > 0$ in which case $\dmatch(B_1, B_2)^2 \gtrsim (\nu_n / n)^2$ for large $n$, hence $\Th_n = \Omega(N \nu_n^2)$. A set of sufficient conditions for consistency are 
$
\min_{r=1,2} \xi_r  = o_p(1)
$ and $\gamma_n = o_p(1)$.

\begin{rem}[Can joint community detection help?]
In our Algorithm~\ref{alg:sbm-ts},
we estimate the labels $\zh_{rt}$  separately for each adjacency matrix $A_{rt}$, and as a result the misclassification rate $\epst$, referenced in both Theorem~\ref{thm:null:main} and~\ref{thm:test:consist}, is at best the optimal rate of ``individual'' community detection given in~\eqref{eq:optimal:miss:rate}. One might wonder if a form of joint community detection across all networks in the sample could significantly improve this rate.  We argue that the answer is negative. In the setting we consider in this paper, 
the networks within a sample only share information via $B_r$ and $\pi_r$.
Thus,
 an oracle that knows $(B_r, \pi_r)$ gains no advantage from knowing the labels of other networks, as they are independent from one network to the next.
The only improvement such oracle can achieve in~\eqref{eq:optimal:matching} is removing the $o(1)$ term in the exponent. In particular, it is impossible to achieve a misclassification rate that improves as a function of $N_r$, the number of networks in the sample.
\end{rem}

\section{Experimental results}
\label{sec:experiments}
In this section, we provide experimental results on real and simulated networks and compare our proposed test to two existing approaches.

\subsection{Competing methods}

We begin by describing the two tests from the literature against which we benchmark our method:  
    (1)  Network Clustering based on Log Moments (NCLM)~\cite{mukherjee2017clustering} and (2) the Maximum Mean Discrepancy of the Adjacency Spectral Embeddings (ASE-MMD)~\cite{tang2017}. It is worth noting that NCLM is designed as a distance between two adjacency matrices, originally proposed for clustering networks. ASE-MMD can be viewed similarly as a distance between two adjacency matrices, designed for two-sample testing based on samples of size 1. However, there is a natural way of extending these ideas to two-sample testing with multiple networks in each sample, which we will discuss at the end of this section. Throughout this section, we use the terms ``network'' and ``adjacency matrix'' interchangeably, since the resulting distance measures between adjacency matrices will be invariant to node permutations.

To measure a distance between two adjacency matrices, 
NCLM constructs a feature vector for each graph
based on its so-called log-moments. To be more precise, \cite{mukherjee2017clustering} considers the $k$-th \emph{graph moment} of a matrix $A$,
   $m_k(A) = \tr[(A/n)^k]$,
which is
the normalized count of closed walks of length $k$. The feature vector for an adjacency matrix $A$
is then defined as 
\begin{equation*}
    g_J(A) := \bigl(\log m_j(A), \, j \in [J] \bigr) 
\end{equation*}
where $J$ is some positive integer. 
The test statistic for  comparing two adjacency matrices $A_1$ and $A_2$
is naturally given as the  $\ell_2$-distance between the feature vectors of the two graphs:
\begin{equation*}
    d_{\text{NCLM}}(A_1, A_2) := \norm{g_J(A_1)- g_J(A_2)}_2.
\end{equation*}
Our experiments, reported in Appendix~\ref{app:choice:of:J}, suggest that a larger value of $J$ improves performance. However, increasing $J$ quickly increases the computation cost. In the results to follow, we have chosen $J=20$ which provides a reasonable balance between performance and cost.

To form a distance between two adjacency matrices, ASE-MMD first computes the so-called adjacency spectral embedding (ASE) for each matrix. 
For an adjacency matrix $A$, 
consider $|A| = (A^{\top} A)^{1/2}$ and let $|A| = \sum_{i=1}^n \lambda_i u_i u_i^{\top}$ be its eigen-decomposition where $\lambda_1 \ge \cdots \ge \lambda_n \ge 0$ are the eigenvalues and $\{u_i\}$ the corresponding orthonormal basis of eigenvectors.
Then, the adjacency spectral embedding of $A \in \{0,1\}^{n \times n}$ into $\reals^d$ is given by
\begin{equation}
    \hat{X}(A) = U_A^{} \sqrt{S_A} \in \reals^{n \times d}
\end{equation}
where $S_A = \diag(\lambda_1, \lambda_2, \dots, \lambda_d)$ and $U_A$ is $n \times d$ matrix whose columns are $u_1, u_2, \dots, u_d$. 
The rows of $\hat X (A)$ define an empirical distribution $\pr_{\hat X(A)} := \frac1n \sum_{i=1}^n \delta_{\hat X_{i*}(A)}$ where $\delta_x$ is the point mass as $x$. ASE-MDD then measures the distance between two adjacency matrices $A_1$ and $A_2$ as the maximum mean discrepancy of the corresponding empirical distributions:
\[
d_{\text{ASE-MMD}}(A_1, A_2) = \text{MMD}(\pr_{\hat X(A_1)}, \pr_{\hat X(A_2)}).
\]
The MMD relies on a positive definite kernel function $\kappa: \reals^d \times \reals^d \to \reals$. Letting $\hat X = \hat X(A_1)$ and $\hat Y = \hat X(A_2)$, one has
\begin{align*}
&\text{MMD}(\pr_{\hat X}, \pr_{\hat Y})  =
\\
&\qquad \frac{1}{n(n-1)} \sum_{i \neq j} \kappa(\hat{X}_{i*}, \hat{X}_{j*}) - \frac{2}{mn} \sum_{i \neq j} \kappa(\hat{X}_{i*}, \hat{Y}_{j*}) + \frac{1}{n(n-1)} \sum_{i \neq j} \kappa(\hat{Y}_{i *}, \hat{Y}_{j*}).   
\end{align*}
In experiments reported here, we consider a Gaussian kernel and use random Fourier features to approximate the MMD. The bandwidth is set to be $\sigma^2 = 1$. Additional experiments, reported in Appendix~\ref{sec:appendix-bw}, show that the results are not sensitive to the choice of bandwidth.

\paragraph{Adapting to multiple samples.}
Next, we describe how we adapt the test statistics to the cases where the sample size is greater than 1. Any measure of dissimilarity, $d(\cdot,\cdot)$, between two networks, can be generalized to the two-sample case, by averaging the dissimilarity 
of pairs of networks from different samples. More specifically,
given the two samples $A_{rt}, \, t \in [N_r]$ for $r=1,2$, we have the two-sample test statistic
$\frac1{N_1 N_2} \sum_{t=1}^{N_1} \sum_{s=1}^{N_2} d(A_{1t},A_{2s})$.

\paragraph{A note on network generation.} Throughout the experiments, we replace Bernoulli variables with a clipped version. In particular, write $X \sim \bernb(p)$ for $p \in \reals$ to denote $X = 1\{ U < p\}$ for some $U \sim [0,1]$. When $p \in [0,1]$, $\bernb(p) = \bern(p)$, but otherwise $\bernb(p)$ is naturally clipped to either 0 or 1. When we write  $\sbm(B,\pi)$ in experiments, it is based on $\bernb(p)$ generation.

\subsection{Choice of $K$ in practice}
\label{sec:choice:of:K}
For our theoretical developments, we assume $K$ to be known. This is not an issue since there are many consistent estimators of $K$, given that the networks are generated from a $K$-SBM; we refer to~\cite{ZhangAmini2023} and the references therein.

In practice, as we show with numerous experiments  below, the test can be used even if the networks are not generated from a $K$-SBM. There, for any fixed choice of $K$, the test can be applied, measuring how close the $K$-level histograms of the two distributions are. Larger $K$ should provide a better approximation (lower bias), but at the same time the resulting SBMs are harder to estimate and match, leading to more variability. To set an optimal $K$ in practice, we propose a procedure akin to parametric bootstrap: 
\begin{enumerate}
    \item For a large value of $K_0$, fit a $K_0$-SBM to one of the samples, resulting in estimates $(\widehat B, \widehat \pi)$.
    \item Perturb $\Bh$ to get $\Bh_{alt}$ by adding, for example, i.i.d. noise $\sim a \cdot \maxnorm{\Bh}\cdot \unif(0,1)$ to each entry of $\Bh$ where $a \in (0,1)$.
    \item Generate many instances of the testing problem $\sbm(\Bh, \widehat\pi)$ vs. $\sbm(\Bh_{alt}, \widehat\pi)$ and for $K \in [K_{\max}]$ estimate the area under the curve (AUC) of the operator receiver characteristic (ROC) associated with SBM-TS test with parameter $K$.
    \item Select the $K$ that leads to the largest AUC.
\end{enumerate}
In our experiments, the choice $K_0 = K_{\max} = 10$ and $a = 0.1$ has worked well.

\subsection{Simple $2$-SBM} Our first simulation reproduces the experiment from~\cite{tang2017} on a 2-SBM.
In particular, we consider testing
$H_0: A, A' \sim \sbm(B_0, \pi)$
versus
$H_1: A \sim \sbm(B_0, \pi), \; A' \sim \sbm(B_{\eps}, \pi)$,
where
\begin{equation}
    B_{\varepsilon} = \begin{bmatrix}
0.5 + \varepsilon & 0.2 \\
0.2 & 0.5 + \varepsilon 
\end{bmatrix},
\end{equation}
and $\pi = (0.4, 0.6)$. Note that each sample contains a single adjacency matrix. We consider vertex sizes $n \in \{100, 200, 500, 1000\}$ and noise levels $\varepsilon \in \{0.01, 0.02, 0.05, 0.1\}$ and evaluate the power of the our SBMTS using 1000 Monte-Carlo replications. Table~\ref{tab:2sbm} summarizes the results alongside the power estimates for ASE-MMD reported   in~\cite{tang2017}. The results clearly show the superior performance of SBMTS along both dimensions $(n,\eps)$.

\begin{table}[t]
\centering
\caption{Power estimates for ASE and SBM-TS for a simple 2-block SBM experiment}
\label{tab:2sbm}
\begin{tabular}{c|cc|cc|cc|cc}
\toprule
 & \multicolumn{2}{c|}{\(\eps = 0.01\)} & \multicolumn{2}{c|}{\(\eps = 0.02\)} & \multicolumn{2}{c|}{\(\eps = 0.05\)} & \multicolumn{2}{c}{\(\eps = 0.1\)} \\
\( n \) & SBM-TS & ASE & SBM-TS & ASE & SBM-TS & ASE & SBM-TS & ASE \\
\midrule
100 & 0.047 & - & 0.161 & 0.06 & 0.776 & 0.09 & 1 & 0.27 \\
200 & 0.15 & - & 0.599 & 0.09 & 1 & 0.17 & 1 & 0.83 \\
500 & 0.785 & - & 1 & 0.01 & 1 & 0.43 & 1 & 1 \\
1000 & 1 & 0.14 & 1 & 1 & 1 & 1 & 1 & 1 \\
\bottomrule
\end{tabular}
\end{table}

\subsection{General SBM with random $B$}
\label{sec:exp-random-B}

To investigate the ``typical'' performance of the test under a general SBM, we generate instances of the testing problem~\eqref{eqn:test2} for a random collection of $(B_1,B_2)$. In particular, we generate 50 random connectivity matrices $B^{(i)}, i \in [50]$, each with entries drawn from $\unif(0.2,0.7)$ subject to symmetry. We then generate corresponding perturbed matrices 
$B_\eps^{(i)} = B^{(i)} + M_\eps^{(i)}$
where $M_\eps^{(i)}$ has entries $N(0, \eps^2)$ subject to symmetry.
For each $i$, we consider the testing problem~\eqref{eqn:test2} for $B_1 = B_2 = \rho B^{(i)}$ versus $(B_1,B_2) = (\rho B^{(i)}, \rho B^{(i)}_\eps)$ where $\rho$ is a sparsity parameter, and we consider a uniform class prior, $\pi = (1/K,\dots,1/K)$.

We set the sample sizes to $N_r = 100$, the number of vertices to $n_{rt} = 10000$, the noise level to $\varepsilon = 0.05$, and sparsity factor to $\rho = 0.1$. For each method, we estimate a mean receiver operating characteristic (ROC) curve across the 50 experiments. More precisely, we replicate each experiment 100 times and obtain 50 ROC curves; then for each false positive rate (FPR), we report the average true positive rate (TPR) across the 50 ROCs. The resulting  mean ROC curves are shown in~Figure \ref{fig:random_B} in Appendix~\ref{app:mean:ROCs}. A summary of these curves via their area under the curve (AUC) is provided in Table~\ref{tab:general_sbm}.

From Table~\ref{tab:general_sbm}, 
we can see that SBM-TS exhibits superior performance over the competitors, but its performance somewhat deteriorates as $K$ increases. 
The primary reason is that significantly increasing $K$  while keeping the number of vertices constant makes the connectivity matrix estimates, $\Bh_r$, noisier; thus, it becomes more challenging to satisfy the approximate matching condition~\eqref{eq:eps:delta:assump}.
Nonetheless, the median area under the ROC curve (AUC) for $K \in \{15, 20\}$ is $0.89$ and $0.87$ respectively, which shows the strong performance of our proposed test in the cases where \eqref{eq:eps:delta:assump} holds. We also note that ASE and NCLM are essentially powerless (AUC $\approx 0.5$) for large $K \in \{15,20\}$.

\begin{table}[t]
\centering
\caption{Mean AUC-ROC for various methods for the general SBM experiment}
\label{tab:general_sbm}
\begin{tabular}{@{}ccccc@{}}
\toprule
\( K \) & 2 & 3 & 15 & 20 \\
\midrule
SBM-TS & 0.95 & 0.91 & 0.83 & 0.81 \\
ASE & 0.62 & 0.54 & 0.49 & 0.48 \\
NCLM & 0.70 & 0.64 & 0.50 & 0.50 \\
\bottomrule
\end{tabular}
\end{table}

\subsection{RDPG}
\label{sec:rdpg}
To demonstrate the performance of our test outside the SBM family, we first consider random dot product graphs (RDPG), 
a class of latent position network models where the probability of an edge formation between two vertices is determined by the dot product of their associated latent positions~\cite{young2007random,JMLR:v18:17-448}. 
Consider latent positions  $\{X_i\}_{i=1}^n$ drawn i.i.d. from some distribution $\mathbb F$, and 
given $\{X_i\}$, generate an adjacency matrix $A$ as
\begin{align*}
      A_{ji} = A_{ij}  \; \sim \overline{\bern} (\rho X_i^\top X_j), \quad \text{for}\; i < j.
\end{align*}
We denote such model as $A \sim \rdpg_n (\mathbb F, \rho)$. We consider two experiments around the two sample testing problem
\[
H_0: A, A' \sim \rdpg_n (\mathbb F, \rho), \quad \text{against}\quad 
H_1: A \sim \rdpg_n (\mathbb F, \rho),\; A' \sim\rdpg_n (\mathbb F', \rho)
\]
We 
set the number of vertices to $n = 10000$ and the sparsity parameter $\rho = 0.15$.

For Experiment~1, we take $\mathbb F = \mathbb F_1 := N(0,\Sigma)$ and $\mathbb F' = N(0,I_2)$ where
\begin{equation*}
    \Sigma = \begin{bmatrix}
        3 & 2\\
        2 & 3
    \end{bmatrix}.
\end{equation*}
For Experiment 2, $\mathbb F =  \mathbb F_2 := \frac12 N(0, \Sigma) + \frac12 N(0, O \Sigma O^\top)$ while $\mathbb F' = N(0,I_2)$ as in Experiment~1. Here, $\mathbb F_2$ is  a mixture of two Gaussian distributions and $O$ is an orthogonal matrix that corresponds to a $\pi/2$ rotation counter-clockwise, namely,
$   
O = \bigl( \begin{smallmatrix}
    0 & 1 \\
    -1 & 0
\end{smallmatrix} \bigr).
$ 

Due to the invariance of the RDPG model to an orthogonal transformation of the underlying node positions, we in fact have $\rdpg_n(\mathbb F_1, \rho) = \rdpg_n(\mathbb F_2, \rho)$, which implies that the two experiments are equivalent.

\begin{figure}[t]
     \centering
     \includegraphics[width=.49\textwidth]{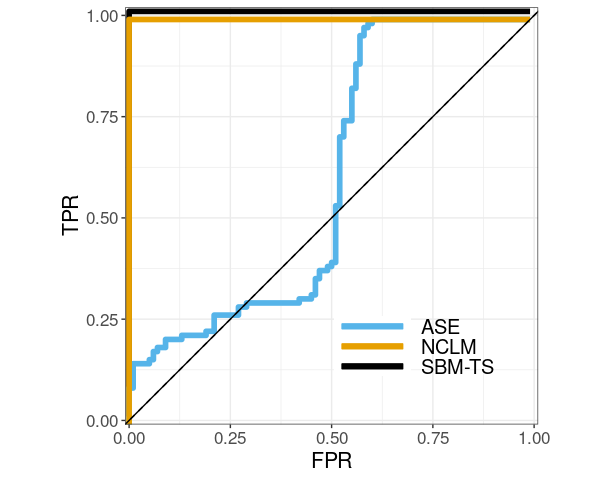}
     \includegraphics[width=.49\textwidth]{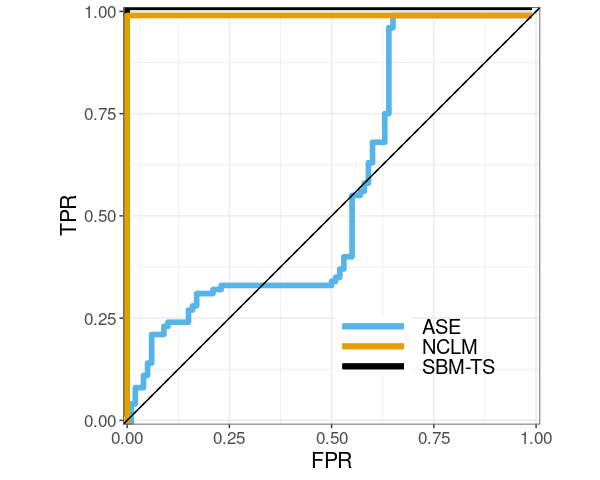}
    \caption{ROC curves for the RDPG Experiment 1 (left) and Experiment 2 (right).}
        \label{fig:rdpg_roc}
\end{figure}

Figure~\ref{fig:rdpg_roc} summarises the results, comparing the performance of SBM-TS to ASE-MMD. We see qualitatively similar behavior for the two experiments, confirming the equivalence of the two models. While SBM-TS essentially exhibits the ROC of the perfect test, ASE-MMD shows poor performance. The reason is that, as proposed in~\cite{tang2017}, ASE-MMD does not have a
matching mechanism to account for the potential orthogonal transformation mismatch between the estimated latent positions for $A$ and $A'$ under null. As a result, even in Experiment~1, the estimated distributions of the latent positions for $A$ and $A'$ could be rotated versions of each other, hence give a large value of MMD. This leads to a bimodal distribution under null for ASE-MMD with one mode essentially overlapping or exceeding the distribution under the alternative, leading to the phase transition behavior observed in the ROC. In contrast, SBM-TS accounts for this potential rotation mismatch by properly matching the connectivity matrices. That is, although we use the same spectral clustering algorithm to estimate the labels for both methods, only ASE-MMD suffers from the orthgonal transformation mismatch. We note that this defect is present in the original paper~\cite{tang2017} as can be seen from the presence of the orthogonal matrix $\bm W_{n,m}$ in the asymptotic null limit of the statistic in Theorem 5 of \cite{tang2017}.

\subsection{A Graphon}
We now consider a more general graphon model. A graphon is a symmetric bivariate function $W:[0,1]^2 \to [0,1]$. We say that $A \sim \graphon(W)$, if for each vertex $i \in [n]$, there is a latent variable $v_i \sim \unif(0,1)$ and given $(v_i)$, 
\begin{align*}
     A_{ij} \sim \bern (W(v_i, v_j)) \, , 1 \leq i < j.
\end{align*}
Consider the two-sample testing problem
\begin{align*}
    H_0: A, A' \sim \graphon(\rho W), \quad \text{against}\quad 
H_1: A \sim \graphon(\rho W),\; A' \sim \graphon(\rho W_{\eps,\delta})
\end{align*}
where $\rho$ is a sparsity factor. We take $W(v_1, v_2) = \frac{1}{4} ( v_1^2 + v_2^2 + \sqrt{v_1} + \sqrt{v_2})$, the graphon from \cite{airoldi2013stochastic},
and let $W_{\eps, \delta}$ be the following perturbation of $W$,
\begin{equation*}
    W_{\varepsilon, \delta}(v_i, v_j) = W(v_i, v_j) +  \varepsilon \cdot 1\bigl\{v_i, v_j \in [0.5 - \delta; 0.5 + \delta] \bigr\}.
\end{equation*}
 In this experiment, we let $\eps = 0.05$, $\delta = 0.2$. For a general graphon, spectral clustering does not provide a good SBM fit. To fit a proper $K$-SBM, we use a Bayesian community detection algorithm~\cite{shen2022bayesian}. The algorithm of Section~\ref{sec:choice:of:K} suggests $K=2$ as the optimal choice and we set $d = K$ for ASE-MMD. Figure~\ref{fig:graphon_roc} illustrate the resulting ROCs for two cases of $(n, \rho)$, namely, $(1000, 0.5)$ and $(10^4, 0.05)$, showing a clear advantage for SBM-TS against the ASE-MMD.

\begin{figure}[t]
     \centering
     \begin{subfigure}{0.45\textwidth}
         \centering
         \includegraphics[width=\textwidth]{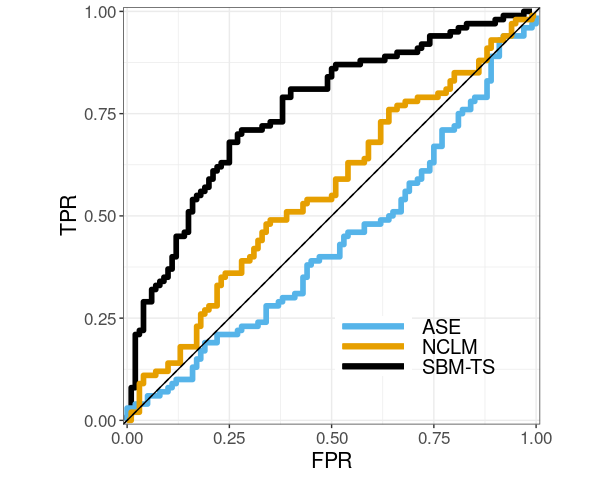}
         \caption{$n = 1000$, $\rho = 0.5$.}
     \end{subfigure}
     \hfill
     \begin{subfigure}{0.45\textwidth}
         \centering
         \includegraphics[width=\textwidth]{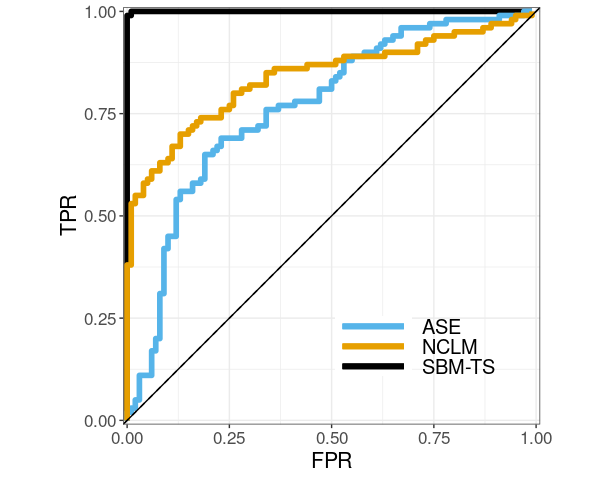}
         \caption{$n = 10000$, $\rho = 0.05$.}
     \end{subfigure}
    \caption{ROC curves for the graphon example.}
        \label{fig:graphon_roc}
\end{figure}

\subsection{COLLAB dataset}
\begin{figure}[t]
     \centering
     \begin{subfigure}{0.24\textwidth}
         \centering
         \includegraphics[width=\textwidth]{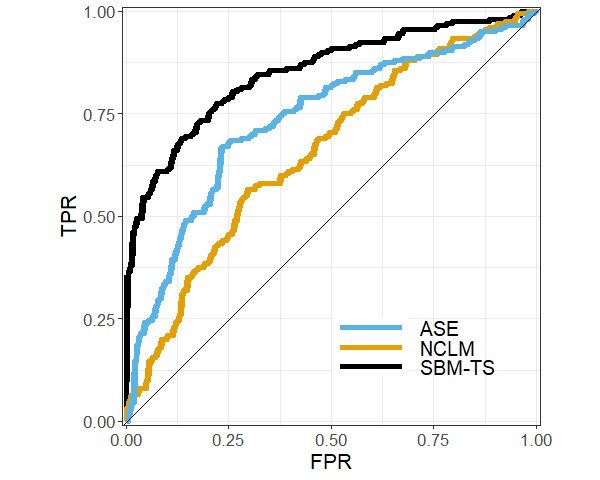}
         \caption{$m = 5$, C1 vs. C3.}
     \end{subfigure}
     \begin{subfigure}{0.24\textwidth}
         \centering
         \includegraphics[width=\textwidth]{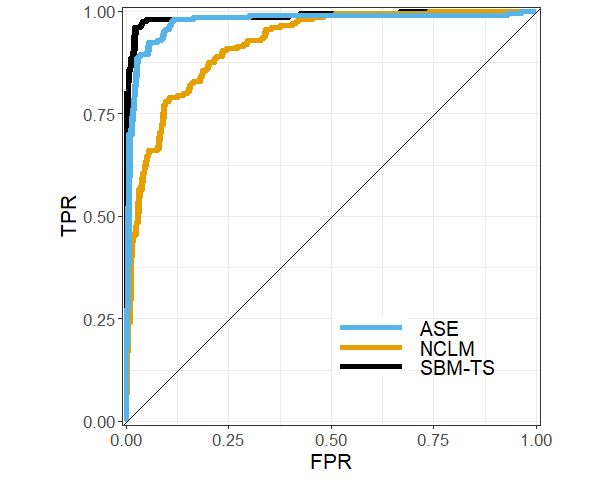}
         \caption{$m = 5$, C2 vs C3.}
     \end{subfigure}
     \centering
     \begin{subfigure}{0.24\textwidth}
         \centering
         \includegraphics[width=\textwidth]{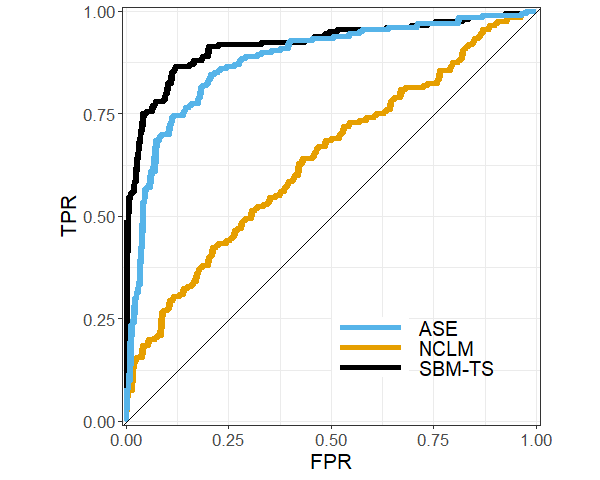}
         \caption{$m = 10$, C1 vs C3.}
     \end{subfigure}
     \begin{subfigure}{0.24\textwidth}
         \centering
         \includegraphics[width=\textwidth]{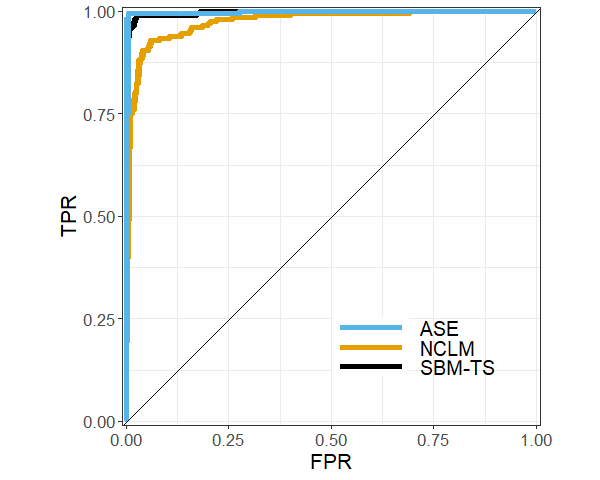}
         \caption{$m = 10$, C2 vs C3.}
     \end{subfigure}
     \caption{ROC curves for the COLLAB dataset.}
     \label{fig:collab_dataset_ROC}
\end{figure}
The COLLAB dataset is a scientific collaboration dataset first introduced in \cite{yanardag}. This dataset contains 5000 networks
derived from
public collaboration data in three scientific fields:
\begin{enumerate*}[$\mathcal C_\arabic*$)]
    \item High energy Physics, \item Condensed Matter Physics, and
    \item Astrophysics.
\end{enumerate*}
Each graph corresponds to an ego-network of a researcher and is labeled by their primary field of research.
On average the graphs in the dataset have around 75 vertices and 2000 edges. Figure \ref{fig:collab:graphs} in Appendix~\ref{app:sample:graphs} demonstrates randomly sampled graphs from this dataset from the three different classes.

Our goal is to evaluate the ability of the test to distinguish between graphs from different classes ($\class_1, \class_2$ and $\class_3$).
We consider two-sample testing problems of the form~\eqref{eqn:test} with $N_1 = N_2 = m$; the two samples under null will be drawn at random (without replacement) both from class $\class_i$, and under the alternative from classes $\class_i$ and $\class_j$ for $i \neq j$. Two choices of $m \in \{5,10\}$ and several choices of $(i,j)$ are considered, as detained in Figure~\ref{fig:collab_dataset_ROC} where the resulting ROCs are plotted.  One observes that SBM-TS has superior or comparable performance to the competing methods in distinguishing the two classes in each case.

\subsection{SW--GOT dataset}
The Star Wars (SW)--Game of Thrones (GOT) dataset \cite{swgot} is derived from popular films and television series. We consider 13 networks: six from the original and sequel SW trilogies, and seven from each of the GOT series. In each graph, vertices
correspond to 
characters and
edges indicate whether two characters share a scene. Let us denote the SW and GOT networks as class $\class_1$ and $\class_2$, respectively. The set of characters for both classes overlap across multiple networks, but no vertex correspondence is utilized because network vertices are unlabeled and each network has different number of vertices.
Sample graphs from this dataset are shown in Figure~\ref{fig:swgot_graphs} in Appendix~\ref{app:sample:graphs}. 

Similar to the COLLAB dataset, we consider a two sample testing problem for distinguishing a null of $(\class_1,\class_1)$ vs. an alternative of $(\class_1, \class_2)$. The resulting ROCs are shown in Figure~\ref{fig:swgot_roc} showing a significant advantage for SBM-TS compared to the competitors.

\begin{figure}[t]
    \centering
    \includegraphics[width = 0.45 \textwidth]{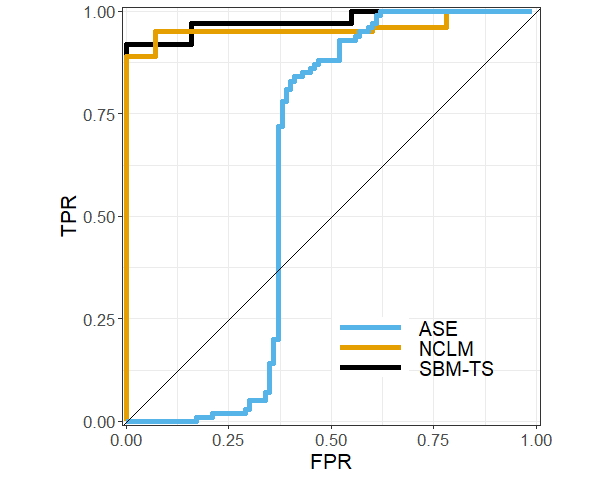}
    \caption{ROC curves for the SW-GOT dataset.}
    \label{fig:swgot_roc}
\end{figure}

\section{Discussion}
\label{sec:discuss}

We presented a flexible test statistic (SBM-TS) for the inference problem of two sample testing for the graphs generated by the stochastic block model. We addressed the main challenge of a true graph level test, where there is no node correspondence, by developing a proper matching of connectivity matrices.
We derived a limiting null distribution result and provided asymptotic power guarantees. Our experiments on a wide range of random graph models, including SBMs, random dot product graphs, and graphons as well as real data examples show that SBM-TS is a computationally efficient and practically viable inference procedure. 

The algorithm for obtaining the two-sample test statistic is modular in a sense that it consists of two independent steps: recovering and matching the connectivity matrices and then forming the test statistic itself based on the discrepancy between the matched matrices. This suggests that there is room for potential development of the algorithm if one is able to improve upon either the matching process or the test statistic construction. For instance, our matching method revolves around the following optimization problem suggested by Lemma~\ref{lemma:q2 = pq1s}:
\begin{equation}\label{eq:opt_problem}
    \argmin_{P  \in \Pi_K, S \in \Psi_K} \fnorm{Q_2 - P Q_1 S}^2 = \argmax_{P \in \Pi_K, S \in \Psi_K  } \tr(Q_2^\top P Q_1 S).
\end{equation}
A potential way to solve this is using alternating minimization algorithm over the permutations matrices $P$ and the sign matrices $S$ since the objective function is separately convex in both arguments. Alternatively, one could employ a low-rank matching method, where the permutation is recovered by matching leading eigenvectors instead of full matrices. More specifically, let $k$-th column of $Q_r$ be denoted as $Q_{rk}$. We note that~\eqref{eq:Q1:Q2} is equivalent to
\begin{align}
    Q_{2k} = P^* Q_{1k} S^*_{kk}, \quad k \in [K].
\end{align}
Since $S^*_{kk} \in \{\pm1\}$, this suggests picking some $k$ and solving
\begin{align}\label{eq:lap:single:sign}
    \argmin_{P \in \Pi_K,\; s \in \{\pm1\}} \norm{P Q_{1k} s - Q_{2k}}^2 =
    \argmax_{P \in \Pi_K,\; s \in \{\pm1\}} \tr(P s Q_{1k} Q_{2k}^T).
\end{align}
That is, we solve the two problems $\lap(Q_{1k} Q_{2k}^T)$ and $\lap(-Q_{1k} Q_{2k}^T)$ and pick a permutation that minimizes the cost in~\eqref{eq:lap:single:sign}. The perturbation analysis suggests that as long as the entries of $Q_{1k}$ are separated well-enough, the permutation matrix $P^*$ has consistent recovery. Finally, we note that for the easier two-sample problem where there is node correspondence among all the networks, it is possible to propose matching algorithms that avoid $(\theta,\eta)$-friendly assumption. 

The choice of the stochastic block model is partially motivated by its ability to approximate more general generative models such as graphons \cite{airoldi2013stochastic, olhede2014network}. In our experiments, we have shown that the SBM-TS is a practical option for two-sample testing of the graphons. Therefore, one of the future directions would be developing theory for the limiting null distribution of the test under graphons that are well-approximated by the stochastic block models. 

\section{Proofs of the main results}

\subsection{Proof of Theorem~\ref{thm:null:main}}
Since $K$ is fixed, we can take $\theta = \Omega(1)$ small enough so that $\theta < \frac1{4 \sqrt K}$. Take $n$ large enough so that~\eqref{assu:simple:scaling} holds. Let $\Dc$ be the event on which~\eqref{assu:gamma} holds.  Since by assumption, $\gamma_n = o_p(1)$ and 
\[
\frac{n}{\nu_n} \cdot \frac{\theta \eta}{2 \sqrt 2 K} \gtrsim 1
\]
we have $\pr(\Dc) = 1 - o(1)$.
Let $B_r^*$ be some (a priori) fixed version of the connectivity matrix for each of the two groups $r=1,2$. By Proposition~\ref{prop:conn:consist} and the union bound, 
there is an event $\Ac$ with
\[
\pr (\Ac^c) \le 3 N_+ K^2 n^{-\alpha} + K N_+ e^{-\kappa^2 n \pmin/3}
\]
and
permutation matrices $\Pt_{rt} \in \Pi_K$ such that on $\Ac \cap \Dc$:
\begin{align*}
    \maxnorm{\Pt_{rt}\Bh_{rt} \Pt_{rt}^{\top} - B^*_r}  &\,\leq\, 
  C_1  \frac{\nu_n}n \Bigl( 56 \beta^2 \cdot\epst_{rt} + \delta_n^{(1)} \Bigr)  \;\le\;  \frac{\nu_n}n \gamma_n \;\le\; \frac{\theta \eta}{2 \sqrt{2} K}
\end{align*}
for all $t \in [N_r]$ and $r = 1,2$,
where $\epst_{rt} = \mis(z_{rt}, \zh_{rt})$.
This implies that
\begin{align*}
    \Ac \cap \Dc \;\subset\;
    \Ec_2 := \Bigl\{ 
    \apermto{\Bh_{rt}, \Pt_{rt}, B_r^*}
    \quad \text{for all}\; t=1,\dots,N_r, \; r=1,2\Bigr\}.
\end{align*}
According to Lemma~\ref{lem:randomization}, we can assume that, in the above, $\Pt_{rt}$ is independent of everything else.

\begin{figure}[t]
	\centering
	\begin{subfigure}[b]{0.3\textwidth}
		\centering
		\begin{tikzcd}[column sep=large, row sep=large,transform canvas={scale=1.25}]
			\Bh_{r1} \arrow[dashed]{r}{\Pt_{r1}} 
				& B_r^* \arrow{d}{I_K} \\
			\Bh_{rt} 
			\arrow[dashed]{r}[swap]{\Pt_{rt}} 
			\arrow[bend left=20]{u}[swap]{\Ph_{rt}}
			& B_r^*
		\end{tikzcd}
            \vspace{1.75cm}
		\caption{Step~\ref{step:match:to:first}}
		\label{fig:match:to:first:a}
	\end{subfigure}
	\begin{subfigure}[b]{0.33\textwidth}
		\centering
		\begin{tikzcd}[column sep=large, row sep=large,transform canvas={scale=1.25}]
			B_r & \\
			\Bh_{r1} \arrow[dashed]{r}{\Pt_{r1}} 
			\arrow[dashed]{u}{I_K}
			& B_r^* \arrow{d}{I_K}
			\arrow{ul}[swap]{\Pt_{r1}^\top}
			\\
			\Bh_{rt} 
			\arrow[dashed]{r}[swap]{\Pt_{rt}} 
			\arrow[bend left=20]{u}[swap]{\Ph_{rt}}
			& B_r^*
		\end{tikzcd}
            \vspace{2.75cm}
		\caption{Simplification by introducing $B_r$.}
		\label{fig:match:to:first:b}
	\end{subfigure}
        \begin{subfigure}[b]{0.3\textwidth}
		\centering
		\begin{tikzcd}[column sep=large, row sep=large,transform canvas={scale=1.25}]
			\Bh_{r1} \arrow[dashed]{r}{I_K} 
				& B_r \arrow{d}{I_K} \\
			\Bh_{rt} 
			\arrow[dashed]{r}[swap]{\Pt_{rt}} 
			\arrow[bend left=20]{u}[swap]{\Ph_{rt}}
			& B_r
		\end{tikzcd}
            \vspace{1.75cm}
		\caption{After conditioning on $B_r$}
		\label{fig:match:to:first:c}
	\end{subfigure}
	\caption{Matching to first network in group \( r \). Here, \( t > 1 \) in the bottom level. Bent arrow is an application of the matching algorithm \( \match \). The edges can be reversed in which case the permutation matrix should be replaced with its transpose. Left-side quantities are random, while right-side quantities are deterministic. Dashed and solid arrows correspond to approximate and exact matching, respectively. See the discussion at the end of Section~\ref{sec:match:consist} for more details on the nature of these diagrams.}
	\label{fig:match:to:first}
\end{figure}
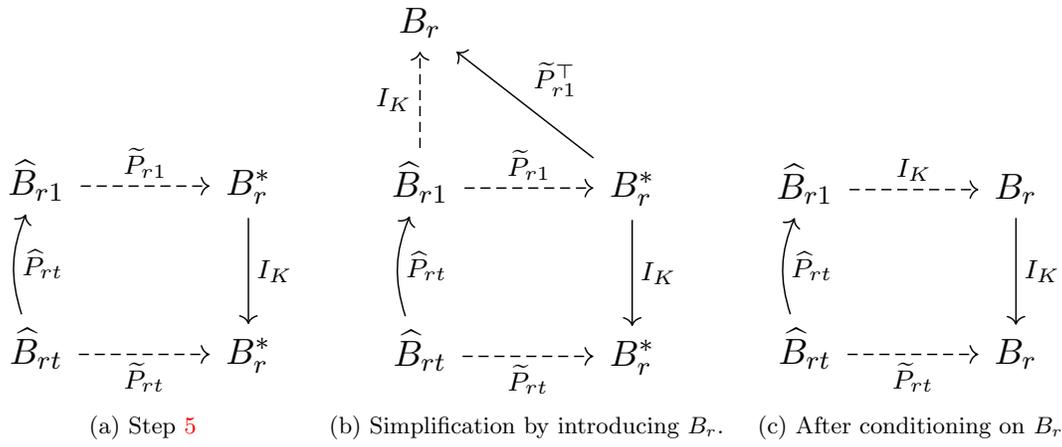

Figure~\ref{fig:match:to:first:a} illustrates the ``matching-to-first'' in step~\ref{step:match:to:first} of the algorithm. As this diagram shows, by Theorem~\ref{prop:matching:recovery}, on event $\Ec_2$, we have $\Ph_{rt}^{\top} = \Pt_{rt}^\top I_K \Pt_{r1}$ that is, 
\begin{align*}
\Ec_2 \subset \Ec_1 := 
    \bigl\{
        \Ph_{rt} =  \Pt_{r1}^\top \Pt_{rt}
    \bigr\}.
\end{align*}
To simplify, let us define $B_r := \Pt_{r1}^\top B_r^* \Pt_{r1}$, that is,
\begin{align*}
    \permto{B_r^*,\Pt_{r1}^\top,B_r}.
\end{align*}
Note that $B_r$ is random version of the true $B_r^*$, due to the randomness of $\Pt_{r1}$, although, it is independent of everything else. Then, a little algebra shows that on $\Ec_1 \cap \Ec_2$, 
\[
\apermto{\Bh_{rt}, \Ph_{rt}, B_r}
\]
for all $t \in [N_r]$ and $r = 1,2$. Figure~\ref{fig:match:to:first:b} illustrates the above inequality (the red path). Now, since $B_r, r = 1,2$ are independent of everything else, we can condition on them and continue with the argument as if they were deterministic. The resulting diagram is shown in Figure~\ref{fig:match:to:first:c}; the effect is as if we assumed $\Pt_{r1} = I_K$ and $B_r = B_r^*$. The above conditioning argument shows that we can do this without loss of generality.

From now on, we work on $(\Ac \cap \Dc) \cap \Ec_1 \cap \Ec_2  = \Ac \cap \Dc$ (by the above argument), on which we have $\Ph_{rt} = \Pt_{rt}$ as discussed above. Then, from the definition of $\Bh_r$ in 
 step~\ref{step:Bh:r}, we note
\[
\Bh_r - B_r = 
 \frac1{N_r} \sum_{t=1}^{N_r} (\Pt_{rt} \Bh_{rt} \Pt_{rt}^{T} - B_r).
\]
On $\Ec_2$ the $\fnorm{\cdot}$ of each term on the RHS is bounded by $\theta \eta / (2\sqrt 2 K)$, and since a norm is a convex function, the same holds for the LHS. That is,
\begin{align}\label{eq:temp:44}
\apermto{\Bh_r,I_K,B_r}
\end{align}
for $r = 1,2$. Since we are under the null, there is a permutation $P^*$ such that $\permto{B_2,P^*,B_1}$. Combining with~\eqref{eq:temp:44}, we can apply Theorem~\ref{prop:matching:recovery}---with $P_1 = P_2 = I_K$ and the roles of $B_1$ and $B_2$ switched---to conclude that $\Ph = P^*$ in step~\ref{step:global:matching}.

\begin{rem}
We could have assumed $B^*_1 = B^*_2$ in the above argument, since we are under the null. However, when passing to $B_r, r=1,2$ we could lose the equality among $B_1$ and $B_2$ (due to $\Pt_{r1}, r=1,2$ potentially being different). Hence, we do not gain anything by making the assumption $B^*_1 = B^*_2$.
\end{rem}

\begin{rem}
    Everything up to and including~\eqref{eq:temp:44} holds under the alternative $B_1 \neq B_2$ as well. This will be used in the proof of Theorem~\ref{thm:test:consist}.
\end{rem}

\medskip
The following arguments are all on $\Ac$.  Let $S_{rt}= \sums(A_{rt},z_{rt})$, $m_{rt} = \counts(A_{rt}, z_{rt})$ and
\[
\Bt_r := \frac{S_r}{m_r}, \quad S_r := \textstyle \sum_t S_{rt}, \quad m_r := \textstyle\sum_t m_{rt} 
\]
Since $\Ph_{rt} = \Pt_{rt}$, we have
\begin{align*}
    \frac{\Sh_r}{\mh_r} = \frac{\sum_t \Pt_{rt} \Sh_{rt} \Pt_{rt}^{\top}}{\sum_t \Pt_{rt} \mh_{rt} \Pt_{rt}^{\top}}
\end{align*}
where $\Sh_r$ and $\mh_r$ are as defined in step~\ref{step:Shr:mhr}.
Let $N = \max\{N_1,N_2\}$.
Then, from Proposition~\ref{prop:conn:consist:multi} and union bound on $r=1,2$, 
there is an event $\Bc_1$ with 
\begin{align}
\label{eq:B1:comp:prob}
    \pr(\Bc_1^c) \le  2(C_2 N + 2K^2) n^{-\alpha} + 2 N K e^{-\kappa^2 n \pmin/3}
\end{align}
such that on $\Bc_1$, we have
\begin{align}\label{eq:upper:main}
\maxnorm{(\Sh_r/ \mh_r) - \Bt_r} \;\le\;   
         40 \,C_1 \beta^3 \frac{\nu_n}n \epst 
\end{align}
where $\epst = \max_{r,t} \epst_{rt}$. Also note that since $\Ph = P^*$, we have $\Sh'_2 = P^* \Sh_2 P^{*T}$ and $\mh'_2 = P^* \mh_2 P^{*T}$.  Moreover, from Proposition~\ref{prop:conn:consist:multi}, on $\Bc_1$,
\begin{align}\label{eq:Btr:Br}
\maxnorm{\Bt_r - B_r} \le C_1 \frac{\nu_n}{n} \delta_n^{(N_r)} \le C_1 \frac{\nu_n}{n}    
\end{align}
since $\delta_n^{(N_r)}  \le \delta_n^{(1)} \le 1$. Using $\maxnorm{B_r} \le C_1 \nu_n /n$ and triangle inequality, we have
\begin{align}\label{eq:upper:control:Bt:Bhr}
    \maxnorm{\Bt_r} \lesssim \frac{\nu_n}{n}, \quad 
    \maxnorm{\Sh_r / \mh_r} \lesssim \frac{\nu_n}{n}
\end{align}
where we have treated $C_1, C_2$ and $\beta^3$ as constants and absorbed them into $\lesssim$ symbol; this will do from time to time in the rest of the proof.

\medskip
Next, we have
$\maxnorm{m_{rt} -\Pt_{rt} \mh_{rt} \Pt_{rt}^{\top}} 
        \le 6 \beta\, n_{rt}^2 \, \epst_{rt}.$
It follows that 
\begin{align}\label{eq:mr:mhr}
    \maxnorm{m_r - \mh_r} \;\le\;
6 \beta \sum_t n_{rt}^2 \, \epst_{rt} 
\;\le\; 6  C_2^2 \beta^2  N_r n^2 \epst     
\end{align}
where $m_r = \sum_t m_{rt}$. Let $m_2' = P^* m_2 P^{*T}$. Then, the same bound as above holds for $\maxnorm{m_2' - \mh_2'}$.  Let $h$ be the elementwise harmonic mean of $m_1$ and $m_2'$.
 Since $\hh$ is the elementwise harmonic mean of $\mh_1$ and $\mh_2'$, we have
\begin{align}\label{eq:hr:hhr}
\maxnorm{h - \hh} \le 12  C_2\beta^2 N n^2 \epst 
\end{align}
where $N = \max\{N_1, N_2\}$. We also not that since $h$ is elementwise the marmonic mean of $m_r, r = 1,2$, we have 
\begin{align}\label{eq:h:lower:bound}
\min_{k,\ell} h_{k\ell} \ge \min_{k,\ell, r} \;  [m_r]_{k\ell} \ge c N n^2 / (2 \beta^2).
\end{align}
The factor 2 is for handling the case $k = \ell$.

\paragraph{Controlling $\sigh$.} Note that
\[
\Bh = \frac{\Sh_1 + P^* \Sh_2 P^{*}}{\mh_1 + P^* \mh_2 P^{*T}}.
\]
Since $B_1 = P^*  B_2 P^{*T}$, this is essentially an estimator like $\Bh$ in Proposition~\ref{prop:conn:consist:multi} based on an independent sample of size $N_+ := N_1 + N_2$ from $\sbm(B_1, \pi)$. It follows from Proposition~\ref{prop:conn:consist:multi} that
there is an event $\Bc_2$ with 
\begin{align}
    \pr(\Bc_2^c) \le  (C_2 N_+ + 2K^2) n^{-\alpha} + N_+ K e^{-\kappa^2 n \pmin/3}
\end{align}
such that on $\Bc_2$, we have
\[
\maxnorm{\Bh - B_1} \le C_1 \frac{\nu_n}{n} 
\Bigl( 40 \beta^3 \epst + \delta_{n}^{(N_+)} \Bigr) \;\le\;
\frac{\nu_n}n \gamma_n
\]
where the second inequality is by $ \delta_n^{(N_+)} \le \delta_n^{(1)}$.

Let $\sigma^2 = (\sigma^2_{k\ell})$ where $\sigma^2_{k\ell} = B_{1k\ell}(1-B_{1k\ell})$. The function $f(x) = x(1-x)$ has derivative satisfying $|f'(x)| \le 1$ for $x \in [0,1]$, hence $f$ is 1-Lipschitz there implying
\[
\maxnorm{\sigh^2 - \sigma^2} \le \maxnorm{\Bh - B_1}.  \]
Since $\nu_n/n \le B_{1k\ell} \le 0.99$, we have 
$\sigma^2_{k\ell} \ge 0.01 \nu_n / n$. 
Ignoring constants, we have shown
\[
\maxnorm{\Bh - B_1} \le \frac{\nu_n}{n} \gamma_n, \quad \maxnorm{\sigh^2 - \sigma^2}\le \frac{\nu_n}{n} \gamma_n, \quad \min_{k,\ell} \sigh^2_{k\ell} \gtrsim \frac{\nu_n}{n}.
\]

\paragraph{High probability event.}
Let $\Bc = \Bc_1 \cap \Bc_2$.
Then, the event $\Ac \cap \Bc$ has high probability. 
Indeed, we have
    \begin{align*}
        \pr(\Ac^c) &\le  3 N_+ K^2 n^{-\alpha} + K N_+ e^{-\kappa^2 n \pmin/3} \\
        \pr(\Bc^c) &\le 3 (C_2 N_+ + 2K^2) n^{-\alpha} + 3 N_+K e^{-\kappa^2 n \pmin/3}.
    \end{align*}
   Using union bound and further bounding $C_2 N_+ + 2 K^2 \le 4 C_2 N_+ K^2$ and $e^{-\kappa^2 n \pmin / 3} \le n^{-\alpha}$ (by assumption~\eqref{assu:simple:scaling}), we obtain \begin{align}\label{eq:Ac:Bc:prob}
       \pr(\Ac^c \cup \Bc^c) \le 19 N_+ K^2 n^{-\alpha}
   \end{align}
    which goes to zero under the assumption $N_+ K^2 n^{-\alpha} = o(1)$.

\paragraph{Controlling $\Th_n$.}  Define
\begin{align*}
      D &:=  \frac{S_1}{m_1} - \frac{S_2'}{m_2'} , \quad 
      \Dh :=   \frac{\Sh_1}{\mh_1} - \frac{\Sh'_2}{\mh'_2}, \\
      \Eh &:= \frac{    \sqrt{\hh}   }{\sqrt{2}   \sigh}  \Dh, \quad 
      E := \frac{ \sqrt{h}   }{ \sqrt{2} \sigma} D
\end{align*}
where $S'_2 = P^* S_2 P^{*T}$ and $m'_2 = P^* m_2 P^{*T}$. Note also that $S_r / m_r = \Bt_r$.  
Let $\df = K(K+1)/2$. For the rest of the proof, we treat the above matrices as vectors in $\reals^{\df}$ by considering only the elements on and above the diagonal, in a particular order (say rowwise).

We have $\Th_n = \fnorm{\Eh}^2$ on $\Ac \cap \Bc$ and since $\pr(\Ac \cap \Bc) = 1-o(1)$, it is enough to establish $\Eh  \convd   (N(0,1))^{\otimes d}$;
see \cite[Theorem 9.15]{Keener2010Theoretical}.
Clearly,
\begin{equation}
    \label{eqn:e1}
         \Eh      \,=\,  \frac{ \sqrt{ \hh }    }{ \sqrt{h}  } \cdot   \frac{ \sigma } {  \hat{\sigma} }  \cdot  E \,+\,   \frac{ \sqrt{ \hh }    }{ \sqrt{h}  } \cdot   \frac{ \sigma } {  \hat{\sigma} }  \cdot  \frac{ \sqrt{h}  }{\sqrt 2\sigma}(  \Dh  - D ).   
\end{equation}
From~\eqref{eq:upper:main}, with high probability, we have that 
\begin{equation}
   \label{eqn:e2}
     \|   \frac{ \sqrt{h}  }{\sqrt 2 \sigma}(  \Dh  - D )\|_{\max}   \,\lesssim\,  \frac{\sqrt{ N n^2 }  }{( \nu_n/n   )^{1/2}} \frac{\nu_n \epst }{n    }  \,=\, \sqrt{ N n \nu_n  }\epst.
\end{equation}
Next we have
\[
\maxnorm[\Big]{\frac{\sigma^2}{\sigh^2}  - \bm 1_\df} \le \frac{\maxnorm{\sigma^2 - \sigh^2}}{\min_{k,\ell} \sigh_{k,\ell}^2} \lesssim \frac{(\nu_n/n) \gamma_n}{ \nu_n / n} \le \gamma_n
\]
hence
\begin{equation}
    \label{eqn:e3}
   \sigma / \sigh = \bm 1_\df + o_p(1). 
\end{equation}
Similarly,
\begin{align*}
    \maxnorm[\Big]{\frac{\hh}{h}  - \bm 1_\df} \le \frac{\maxnorm{\hh - h}}{\min_{k,\ell} h_{k\ell}} \lesssim \frac{ N n ^2 \epst}{N n^2} \le \epst
\end{align*}
and hence 
\begin{equation}
    \label{eqn:e4}
    \sqrt{\hh} / \sqrt{h} = \bm 1_\df + o_p(1).
\end{equation}

\begin{lem}
\label{lem:normality} 
 $E := ( \sqrt{h} / ( \sqrt{2} \sigma)) D \convd (N(0,1))^{\otimes \df}$, under the assumptions of Theorem~\ref{thm:null:main}. 
\end{lem}
\begin{proof}
    See Section \ref{sec:lem_normality}.
\end{proof}

 Therefore, from  (\ref{eqn:e1})--(\ref{eqn:e4}), Lemma \ref{lem:normality}, and Slutsky's theorem, we obtain that  
$
  \Eh  \convd   (N(0,1))^{\otimes d},
$
provided that $\sqrt{ N \nu_n n} \, \epst   \,=\, o(1)$. The proof is complete.

\subsection{Proof of Lemma~\ref{lem:normality} }
\label{sec:lem_normality}
Recall that the $L^1$ Wasserstein distance between the distributions of two random variables $Y$ and $Z$ can be expressed as
\[
d_{W_1}(Y,Z) = \sup_{h:\;  \lipnorm{h} \le 1} | \ex h(Y) - \ex h(Z)|
\]
where $h$ ranges over all $1$-Lipschitz functions $h : \reals \to \reals$, that is, $h$  that satisfy $|h(x) - h(y)| \le |x - y|$ for all $x, y \in \reals$. See~\cite[Chapter 4]{Chen2010Normal} or~\cite{ross2011fundamentals}.

Lemma \ref{lem:normality} follows form the following results:
\begin{prop}
\label{prop:clt1}
    Let $S \sim \bin(n,p)$ with $p \le 1/2$ and let $W  = \frac{\sqrt n}{\sigma} (\frac{S}{n} - p)$ where $\sigma = \sqrt{p q}$ and $q = 1 - p$. Let $C \le 0.4785$ be the constant in the Berry-Esseen bound. Then,
    \begin{align}
        \sup_{x \in \reals}\, |\pr( W \le x) - \Phi(x) | &\le \frac{C}{\sqrt {np/2}}  \\ \sup_{h: \lipnorm{h} \le 1} \,
        \bigl|\ex h(W) - \ex h(Z) \bigr| &\le \frac{1}{\sqrt{np/2}} 
    \end{align}
    where $Z \sim N(0,1)$.
\end{prop}
\begin{proof}
    We can write $S = \sum_{i=1}^n B_i$ where $B_i$ are i.i.d. Bernoulli($p$) variables.
    Let $X_i = (B_i - p) / \sigma$ be the standardized versions and note that $W = \frac1{\sqrt n} \sum_{i=1}^n X_i$.
    Berry-Esseen bound gives
    \begin{align*}
        \sup_{x \in \reals} |\pr( W \le x) - \Phi(x) | \;\le\; C
        \frac{\ex |X_1|^3}{\sqrt n} 
    \end{align*}
    and Corollary~4.1 of~\cite[page 67]{Chen2010Normal} gives
    $
        d_{W_1}(W,Z) \le \frac{\ex|X_1|^3}{\sqrt n}.
    $
    We have 
    \begin{align*}
        \ex|X_1|^3 = \frac1{\sigma^3} \cdot \ex|B_1 - p|^3 = \frac{p^3 q + q^3 p}{(pq)^{3/2}} = \frac{p^2 + q^2}{\sqrt{pq}} \le \frac{(p+q)^2}{\sqrt{p/2}}
    \end{align*}
    which gives the desired inequalities.
    
\end{proof}

\begin{lem}
\label{lem:deterministic:clt}
Suppose that for $r  \in \{1,2\}$ and $t \in [N_r]$ the symmetric matrices $A_{rt} \in \{0,1\}^{  n_{rt} \times n_{rt} }$ are independent with independent entries on and above diagonal that satisfy
$
 (A_{rt})_{ij} \sim \bern(  B_{( z_{rt})_i,(z_{rt})_{j} } ) 
$
where $z_{rt} \in[K]^{ n_{rt} } $ is deterministic. 
Let 
\[
S_r = \sum_{t=1}^{N_r} \sums(A_{rt}, z_{rt}), \quad 
m_r = \sum_{t=1}^{N_r} \counts(A_{rt}, z_{rt}),
\]
and for all $k,\ell \in [K]$, let $\sigma_{k\ell} = \sqrt{B_{k\ell} (1-B_{k\ell})}$ and  set 
\begin{align}
\label{eq:xi:def}
     \xi_{k\ell }  \,=\,     \frac{\sqrt{\mb_{k\ell}}}{ \sqrt2 \sigma_{k\ell}  } \Bigl( \frac{(S_1)_{k\ell}}{m_{1k\ell}} - \frac{(S_2)_{k \ell}}{m_{2k\ell}} \Bigr)\end{align} 
where $\bar{m}_{k \ell}$  is the harmonic mean of $  (m_1)_{k\ell}$   and  $(m_2)_{k\ell}$.   Assume that $ m_{2k\ell} \le c_1 m_{1 k \ell}$ for some  $c_1 > 0$. 
Then
\begin{align*}
    \sup_{x \in \reals} \bigl| \pr( \xi_{k\ell} \le x) -  \pr (Z \le x) \bigr|  \le C \sqrt{\frac{n}{ \nu_n}} \Bigl( \frac{1}{\sqrt {m_{1k\ell}}} + \frac1{\sqrt{m_{2k \ell}}} \Bigr).
\end{align*}
where $C >0$ is a constant dependent on $c_1$ and $Z \sim N(0,1)$.

\end{lem}
\begin{proof}
Let $\sigma_{kl}= \sqrt{B_{kl}(1-B_{kl}) }$.
First, notice that by Proposition~\ref{prop:clt1}, it holds that 
\[
\underset{x \in \mathbb{R}}{\sup}\,\left\vert \mathbb{P}\left( \frac{  \sqrt{m_{rkl} }   }{  \sigma_{kl}
}\left(\frac{S_{rkl} }{ m_{rkl} } - B_{kl} \right)\leq x \right) - \Phi(x)  \right\vert \,\leq\, \frac{C }{ \sqrt{m_{rkl} \nu_n/n  } }
\]
for some constant $C>0$, and for $r=1,2$. Hence,
\begin{equation}
    \label{eqn:s1}
    \underset{x \in \mathbb{R}}{\sup}\,\left\vert \mathbb{P}\left( 
    \frac{\sqrt{\bar m_{k\ell}}}{\sqrt 2\, \sigma_{kl}} \cdot\frac{ S_{rkl } }{ m_{rkl}  }   \leq x \right) - \Phi\Bigl(
    \sqrt{ \frac{2 m_{rkl}}{\bar m_{k\ell}} }\bigl(x- \frac{\sqrt{\bar m_{k\ell}}B_{kl}}{\sqrt2 \sigma_{kl}}  \bigr)   
    \Bigr) 
    \right\vert \,\leq\, \frac{C }{ \sqrt{m_{rkl} \nu_n/n  } }.
\end{equation}
Define the random variables
\[
X_r \,=\,\frac{\sqrt{\bar m_{k\ell}}}{\sqrt 2 \, \sigma_{kl}} \cdot\frac{ S_{rkl } }{ m_{rkl}  },  \quad r= 1,2,
\]
and consider two independent random variable  $Y_1$ and $Y_2$ with 
\[
Y_r \,\sim N\left(\frac{\sqrt{\bar m_{k\ell}}B_{k\ell}}{\sqrt2 \sigma_{k\ell}}, \,  \frac{\bar m_{k\ell }}{2 m_{rk\ell} } \right).
\]
Notice that
$
  Y_1 - Y_2 \,\sim \, N(0,1).
$
Let $F_Z$ denote the CDF of a random variable $Z$. Then, we can rewrite~\eqref{eqn:s1} as
\begin{align}
    \sup_{x \in \mathbb{R}} 
    | F_{X_1}(x) - F_{Y_1}(x)| 
    \,\leq\, \frac{C }{ \sqrt{m_{rkl} \nu_n/n} }.
\end{align}
For any $x, x_2 \in \reals$, we have
\[
\pr(X_1 - X_2 \le x \given X_2 = x_2) = \pr(X_1 \le x + x_2) = F_{X_1}(x + x_2),
\]
by independence of $X_1$ and $X_2$. Fix $x \in \reals$.
Since $\pr(X_1 - X_2 \le x) = \ex[ \pr(X_1 - X_2 \le x \given X_2)]$,
it follows that
\begin{align*}
    & |\mathbb{P}( X_1 - X_2 \leq x) - \mathbb{P}(Y_1 -Y_2 \leq x) | \\
     &\qquad \quad = 
    \bigl| \ex [F_{X_1}(x + X_2)] - \ex[F_{Y_1}(x + Y_2)] \bigr| \\
    &\qquad \quad =
    \bigl| \ex [F_{X_1}(x + X_2)] - \ex[F_{Y_1}(x+X_2) + \ex[F_{Y_1}(x+X_2) - \ex[F_{Y_1}(x + Y_2)] \bigr|  \\
    &\qquad \quad \le 
    \ex \bigl|F_{X_1}(x + X_2) - F_{Y_1}(x + X_2) \bigr| +  \bigl| \ex[F_{Y_1}(x+X_2) - \ex[F_{Y_1}(x + Y_2)] \bigr| \\
    &\qquad \quad\le
    \frac{C }{ \sqrt{m_{1kl} \nu_n/n} } + 
    | \ex h(X_2) - \ex h(Y_2) |
\end{align*}
where we have defined $h(z) = F_{Y_1}(x + z)$. Since 
\begin{align*}
\lipnorm{h} \le \infnorm{h'} &= \sqrt{ \frac{2 m_{2 kl}}{\bar m_{k\ell}}} \cdot \infnorm{\Phi'} \\  &= \sqrt{1 + \frac{m_{2kl}}{m_{1kl}}} \cdot \frac1{\sqrt{2\pi}} \le \sqrt{1 + c_1}\cdot \frac1{\sqrt{2\pi}} =: c_2    
\end{align*}
by assumption. Then, from Proposition~\ref{prop:clt1}, $| \ex h(X_2) - \ex h(Y_2)| \le c_2 / \sqrt{m_{2kl} \nu_n / n}$
\begin{align*}
    |\mathbb{P}( X_1 - X_2 \leq x) - \mathbb{P}(Y_1 -Y_2 \leq x) | \le 
    \frac{C }{ \sqrt{m_{1kl} \nu_n/n} } + 
    \frac{c_2}{\sqrt{m_{2kl} \nu_n / n}}
\end{align*}
which gives the desired result.
\end{proof}

Let $\mathcal{I}(\reals)$ be the set of indicator functions of half-intervals, that is,
\[
\mathcal I(\reals) = \bigl\{t \mapsto 1\{t \le x\} : x \in \reals \bigr\}. 
\]

\begin{lem}
\label{lem:cond:marginal:to:prod}
    Consider random variables $X_{ni}, i \in [K]$ and $Y_n$ and assume that $\{X_{ni}, i \in [K]\}$ are independent conditional on $Y_n$. In addition, we have
\[
\sup_{h \,\in\, \mathcal I(\reals)} | \ex [h(X_{ni}) \given Y_n] - \ex h(Z)| \cdot 1\{Y_n \in \Ac_n\} 
 \;\le\; \eps_n ,\quad i \in [K]
\]
for some sequence of events $\Ac_n$ and a deterministic sequence of $\eps_n > 0$, and some random variable $Z \sim \mu$. Assume that $\eps_n \to 0$ and $\pr(Y_n \in \Ac_n^c) \to 0$ as $n\to \infty$.  Then 
\[(X_{n1},\dots,X_{nK}) \convd \mu^{\otimes K}.\] 
\end{lem}

\begin{proof}
    Let $Z_i, i \in [K]$ be i.i.d. draws from $\mu$.
    It is enough to show that 
    \begin{align}
        \label{eq:product:BL1}
        \ex \Bigl[\prod_{i=1}^K f_i(X_{ni})\Bigr] \to \ex\Bigl[\prod_{i=1}^K f_i(Z_i)\Bigr] = \prod_i \ex f_i(Z_i)
    \end{align}
    for any collection of $f_1,\dots,f_K \in \mathcal I(\reals)$.
    
    Let us fix one such collection and, for simplicity, write 
    $
    W_n = \prod_i f_i(X_{ni})$
    and $C = \prod_i \ex f_i(Z_i)$.
    We want to show $\ex[W_n| \to C$.
    From the assumption, it follows that
    \begin{align*}
    \bigl| \ex [f_i(X_{ni}) \given Y_n] - \ex f_i(Z_i) \bigr| \;\le\; \eps_n + 2 \cdot 1\{Y_n \in \Ac_n^c\}.
    \end{align*}
    By conditional independence, $\ex [W_n \given Y_n] = \prod_i \ex[f_i(X_{ni})\given Y_n]$. Then, using $|\prod_{i=1}^K a_i - \prod_{i=1}^K b_i| \le K \max_i |a_i - b_i|$, which holds if $a_i, b_i \in [-1,1]$ for all $i$, we have
    \begin{align*}
        \bigl| \ex [W_n \given Y_n ] - C\bigr| \;\le\; K \eps_n  + 2 K \cdot 1\{Y_n \in \Ac_n^c\}.
    \end{align*}
    Then, we have 
    \begin{align*}
        | \ex [W_n] - C|  &= | \ex[\ex[W_n \given Y_n]] - C| \\ &\le \ex \bigl|\ex[W_n \given Y_n]] - C\bigr| \le  K \eps_n  + 2 K \pr(Y_n \in \Ac^c).
    \end{align*}
    Letting $n \to \infty$, the desired result follows from the assumptions.
\end{proof}

\begin{proof}[Proof of Lemma~\ref{lem:normality}]

 Let $z = (z_{rt}, t \in [N_r], r=1,2)$.
 Let $\mathcal M(c_1) = \{z:\; m_{2k\ell} \le c_1 m_{1k\ell}\}$.
 By Lemma~\ref{lem:deterministic:clt}, we have 
\begin{align}
    \label{eq:temp:7}
    \sup_{h \in \mathcal I(\reals)} \bigl|  \ex[ h(\xi_{k\ell}) \given z] -  \ex h(Z) \bigr| \cdot 1\{z \in \mathcal M(c_1)\} \le C(c_1) \sqrt{\frac{n}{ \nu_n}} \Bigl( \frac{1}{\sqrt {m_{1k\ell}}} + \frac1{\sqrt{m_{2k \ell}}} \Bigr)
\end{align}
where $\xi_{k\ell}$ as defined in~\eqref{eq:xi:def}, $Z \sim N(0,1)$ and $C(c_1)$ is some constant dependent on $c_1$.
Let $n_{rtk} := \sum_{i=1}^{n_t} 1\{(z_{rt})_i = k\}$, and consider the event
\[
   \mathcal A_n\,:=\,\Bigl\{  \min_{{r, k}}  \,n_{r tk}\, \geq   n_t / \beta ,\; \forall t \in [N] \Bigr\}.
\]
Then on $\Ac_n$, we have
$m_{rk\ell} \ge c N n^2 /\beta^2$; see also~\eqref{eq:h:lower:bound}. We also have $\pr (\Ac_n^c) = o(1)$ under the assumptions of Theorem~\ref{thm:null:main}, by the same argument that controls $\Ec_1$ in Proposition~\ref{prop:conn:consist:multi}.
Moreover, assuming $N_2 \le N_1$---without loss of generality---on $\Ac_n$, we have
\[
\frac{m_{2k\ell}}{m_{1k\ell}} \le \frac{N_2 (C_2 n)^2}{N_1 n^2 /\beta^2} \le \beta^2 C_2^2
\]
where we have used the size assumption~\eqref{eq:size}. That is, $\Ac_n \subset \{z \in \mathcal M(\beta^2 C_2)\}$.
Taking $c_1 = \beta^2 C_2^2$ in ~\eqref{eq:temp:7} and multiplying both sides of the inequality by $1_{\Ac_n}$, we obtain
\begin{align*}
    \sup_{h \in \mathcal I(\reals)} \bigl|  \ex[ h(\xi_{k\ell}) \given z] -  \ex h(Z) \bigr| \cdot 1_{\Ac_n}  \le \frac{2\beta \,C(\beta^2 C_2)}{\sqrt c} \frac{1}{\sqrt{N n \nu_n}}.
\end{align*}
Since $\{\xi_{k\ell}\}$ are independent given $z$, the result now follows from Lemma~\ref{lem:cond:marginal:to:prod} given $N n \nu_n \to \infty$.
\end{proof}
 
\subsection{Proof of Theorem~\ref{thm:test:consist}}
    We will follow the notation and argument in the proof of Theorem~\ref{thm:null:main}. On event $\Ac \cap \Dc \cap \Bc$ defined there, we have correct matching $\Ph_{rt} = \Pt_{rt}$---where $\Pt_{rt}$ is as defined in the proof of  Theorem~\ref{thm:null:main}---and~\eqref{eq:upper:main} and~\eqref{eq:Btr:Br} hold. Since $\Gc \subset \Dc$, all the above also holds on $\Ac \cap \Gc \cap \Bc$. This is the event we will work on. 
    
    The two inequalities~\eqref{eq:upper:main} and~\eqref{eq:Btr:Br}  give
    \[
    \maxnorm{(\Sh_r / \mh_r) - B_r} \le C_1 \frac{\nu_n}{n} \Bigl( 40 \beta^3 \epst + \delta_n^{(N_r)}   \Bigr) = \xi_r \frac{\nu_n}{n}, \quad r=1,2.
    \]
    Since, $\Sh'_2 / \mh'_2 = \Ph (\Sh_2 / \mh_2) \Ph^T$, we also have 
    \[
     \maxnorm{(\Sh_2' / \mh_2') - \Ph B_2 \Ph^T} \le
     \xi_2 \frac{\nu_n}{n}.
    \]
    Using $(a+b+c)^2 \le 3(a^2 + b^2 +c^2)$, we obtain
    \begin{align*}
        \fnorm{(\Sh_1 / \mh_1) - (\Sh_2' / \mh_2')}^2 &\ge 
        \frac13 \fnorm{B_1 -  \Ph B_2 \Ph^T}^2 \\
        &\quad - \fnorm{(\Sh_1 / \mh_1) -  B_1 }^2
        - \fnorm{(\Sh'_2 / \mh'_2)  -  \Ph B_2 \Ph^T}^2 \\
        &\ge \frac13 \min_{P} \fnorm{B_1 -  P B_2 P^T}^2 -    2K^2  \Bigl( \frac{\nu_n}{n}\Bigr)^2 \max_{r=1,2} \xi_r^2 \\
        &\ge \frac16 \min_{P} \fnorm{B_1 -  P B_2 P^T}^2 
    \end{align*}
    by assumption~\eqref{eq:dmatch:lower}.
    Next, we note that~\eqref{eq:mr:mhr} holds in this case. Let $m'_2 = \Ph m_2 \Ph^T$ and let $h$ be the elementwise harmonic mean of $m_1$ and $m'_2$. Then, both~\eqref{eq:hr:hhr} and~\eqref{eq:h:lower:bound} hold, irrespective of the specific $\Ph$. On $\Gc$, we have $48 C_2 \beta^4 \epst \le 1$ which, combined with the latter two inequalities, yields 
    \[
    \min_{k\ell} \hh_{k\ell} \ge N n^2 /(4\beta^2).
    \]
    We also note that $\sigh_{k\ell}^2 = \Bh_{k\ell}(1-\Bh_{k\ell}) \le 1/4$. Then, \begin{align}
        \Th_n \ge \frac{N n^2 /(4\beta^2)}{2 \cdot \frac14} \fnorm{(\Sh_1 / \mh_1) - (\Sh_2' / \mh_2')}^2.
    \end{align}
    Putting the pieces together, we have the desired inequality on $(\Ac \cap \Bc) \cap \Gc$. Combined with the probability bound~\eqref{eq:Ac:Bc:prob} for $(\Ac \cap \Bc)^c$,
   the proof is complete.

{\small 
\printbibliography
}

\appendix

\section{Consistency of connectivity matrix}
Given an $n \times n$ adjacency matrix $A$, we apply a community detection algorithm to obtain label vector $\zh = (\zh_1,\dots,\zh_n) \in [K]^n$.  
Let $\nhamming$ be the normalized Hamming distance, that is, $\nhamming(z,\zh) = \hamming(z,\zh)/n$ where
$\hamming(z,\zh) = \sum_{i=1}^n 1\{z_i \neq \zh_i\}$ is the Hamming distance.

\begin{remark}
\label{remark_kl}
In the proof of consistency below, our  results are given in terms of $ \maxnorm{ E} $ where $E$ here is a $K\times K$ matrix placeholder for different matrices as given in Propositions~\ref{prop:conn:consist} and \ref{prop:conn:consist:multi}. However, to simplify the proofs below, when computing upper bounds for $\maxnorm{ E}  \,=\, \max_{k,\ell \in \{1,\ldots,K\} }  \vert E_{k\ell}\vert $, we focus on the case $k\neq \ell$ and omit the case $k =\ell$ as it is similar. The only difference in dealing with $k =\ell$ is that the constants would need to be inflated.
\end{remark}

\begin{proposition}\label{prop:conn:consist}
    Assume that $(A,z) \sim \sbm_n(B, \pi)$ with $B$ satisfying~\eqref{eq:sparsity}. Let
    $
    S = \sums(A, z),\;  \Sh = \sums(A, \zh)$ and $
    m = \counts(z),\; \mh = \counts(\zh)$. Set 
    \[
    \Bh = \conn(A, \zh) = \Sh / \mh, \quad \Bt = \conn(A,z) = S / m.
    \]
    Fix $\kappa \in (0,1)$ and $\alpha > 0$, and let $\beta = 1/(\kb \pmin)$ with $\kb = 1-\kappa$, and define 
    \[
    \delta_n := \sqrt{\frac{3 \beta^2 \alpha \log n}{ n \nu_n}}.
    \]
    Assume that $\delta_n \le 1$ and $\nu_n \ge 3 (1+\alpha) \log n$.
    For $\sigma \in \Pi_K$, let $\eps(\sigma) = \nhamming(z, \sigma \circ \zh)$.
    Then,  with probability at least $1 - 3 K^2 n^{-\alpha} - K e^{-\kappa^2 n \pmin/3}$, for all $\sigma \in \Pi_K$ such that $12 \beta^2 \eps(\sigma) \le 1$,
	\begin{align}
	\maxnorm{P_\sigma\Bh P_\sigma^T - \Bt}  &\,\leq\, 
 56 C_1 \beta^2 \cdot \frac{\nu_n}n \eps(\sigma),
         \label{eq:B:Bh:max:dev}\\
         \maxnorm{\Bt - B}  &\,\leq\, 
            C_1 \frac{\nu_n}n \delta_n,
        \label{eq:B:Bt:max:dev}\\
        \maxnorm{S - P_\sigma \Sh P_\sigma^T} &\,\le\, 
        8C_1 \beta \cdot n \nu_n \eps(\sigma),
        \label{eq:S:Sh:max:dev} \\
        \maxnorm{m - P_\sigma \mh P_\sigma^T} &\,\le\, 
        6 \beta
        \cdot n^2 \eps(\sigma), \label{eq:m:mh:max:dev}
        \\
        \min_{k, \ell} m_{k\ell} &\ge n^2 / \beta^2.
	\end{align}
\end{proposition}

\begin{proof}[Proof of Proposition~\ref{prop:conn:consist}]

By redefining $\zh$ to be $\sigma \circ \zh$, one gets $P_\sigma \Bh P_\sigma^T$ in place of $\Bh$. Thus, without loss of generality, we can assume $\sigma$ to be the identity permutation.
Let
\begin{align*}
	d_k(z,\zh) = \sum_{i} \bigl| 1\{z_i = k\} - 1\{\zh_i = k\}\bigr|, \quad d_H(z,\zh) = \sum_i 1\{z_i \neq \zh_i\}
\end{align*}
so that $\sum_k d_k(z, \zh) = 2d_H(z,\zh)$. This can verified by writing
\begin{align*}
	\bigl| 1\{z_i = k\} - 1\{\zh_i = k\}\bigr| &= 1\{z_i = k, \zh_i \neq k\}  + 1\{z_i \neq k, \zh_i = k\}.
\end{align*}
Let  $n_k   \,=\,   \sum_{i=1}^n \ind{z_i = k }$ and $\nh_k = \sum_{i=1}^n \ind{\zh_i = k}$ and note that $|n_k - \nh_k| \le d_k(z,\zh)$. 
Let us also write
\begin{align}\label{eq:rho:def}
	\rho := \max_k \frac{d_k(z,\zh)}{n_k} \end{align}
and note that $|(\nh_k / n_k) - 1| \le \rho$.

Let us consider some events.
First, consider event
\begin{align}
    \Ec_0 := \Bigl\{ \max_j \sum_i A_{ij} \le 2C_1 \nu_n \Bigr\}.
\end{align}
Applying Proposition~\ref{prop:binom:concent} conditioned on $z$, with $p = C_1 \nu_n /n$, and then taking expectation of both sides to remove the conditioning, we have
\begin{align*}
    \pr\Bigl( \sum_i A_{ij}  \ge C_1 \nu_n (1+u) \Bigr) \le e^{-u^2 C_1 \nu_n/3}.
\end{align*}
Take $u = 1/\sqrt{C_1}$. Then, $\pr(\Ec_0^c) \le n^{-\alpha}$ by union bound and assumption $\nu_n /3 \ge (1+\alpha) \log n$.

\medskip
Next, note that $\pr(n_k \le (1-\kappa) n \pi_k) \le \exp(-\kappa^2 n \pi_k/3)$ for $\kappa \in (0,1)$,  by Proposition~\ref{prop:binom:concent}.
Consider the event 
\[
   \mathcal{E}_1\,:=\,\Bigl\{  \underset{k =1,\ldots K}{\min}  \,n_k\, \geq   n / \beta  \Bigr\},
\]
where $1/\beta = (1-\kappa) \pmin$. Then, by union bound 
$
\pr( \Ec_1^c ) \le K \exp(-\kappa^2 n \pmin/3).
$
On $\Ec_1$,
\begin{align}\label{eq:rho:nhamming}
	\rho := \max_k \frac{d_k(z,\zh)}{n_k} \le \frac{2 d_H(z,\zh)}{\kb \pmin n} = 2 \beta \nhamming (z, \zh).
\end{align}

Next, we apply Proposition~\ref{prop:binom:concent} conditioned on $z$, to get
\begin{align*}
	\pr\bigl(|B_{k\ell} -\Bt_{ k \ell} | \ge \delta_n B_{k\ell} \; \given z \bigr) \cdot 1_{\Ec_1} 
	&\;\le\; 2 e^{- \delta_n^2 n_k n_\ell B_{k\ell} / 3} \cdot 1_{\Ec_1}\\
	&\;\le\; 2 e^{- \delta_n^2 (n^2 / \beta^2) (\nu_n / n)/ 3}
\end{align*}
where we have used the lower bound in~\eqref{eq:sparsity}. Take $\delta_n^2 = 3 \beta^2 \alpha \log n / ( n \nu_n)$ and let
\begin{align}
	\Ec_2 = \Bigl\{ | B_{k\ell} - \Bt_{k\ell}| \le \delta_n 
 B_{k\ell} \;\; \text{for all}\; k,\ell \in [K] \Bigr\}.
\end{align}
Then, by union bound, $\pr(\Ec_1 \cap \Ec_2^c) \le \pr(\Ec_2^c \given \Ec_1) \le 2 K^2 n^{-\alpha}$. We will work on $\Ec_0 \cap \Ec_1 \cap \Ec_2$ and  note that 
\[
\pr(\Ec_0 \cap \Ec_1 \cap \Ec_2) \ge 1 - \pr(\Ec_0^c) - \pr(\Ec_1^c) - \pr(\Ec_1 \cap \Ec_2^c).
\]
Note that on $\Ec_2$,
 \begin{align}
     \max_{k,\ell}| B_{k\ell} - \Bt_{k\ell}| 
     &\le \delta_n \cdot C_1 \nu_n /n, 
     \label{eq:B:Bt}\\
     \max_{k,\ell} \Bt_{k\ell} &\le 2 C_1 \nu_n / n  
     \label{eq:Bt:upper:bound}
 \end{align}
 using the upper bound in~\eqref{eq:sparsity} and assumption $\delta_n \le 1$. This proves~\eqref{eq:B:Bt:max:dev}.
 
\medskip

Now, let us establish~\eqref{eq:S:Sh:max:dev}. We have
\begin{align*}
 | S_{k\ell} - \Sh_{k\ell} |
 &\le \sum_{i,j} A_{ij} |  1\{z_i = k, z_j = \ell\} - 1\{\zh_i = k, \zh_j = \ell\} |.
\end{align*}
 Adding and subtracting $1\{\zh_i = k, z_j = \ell\}$ and expanding by triangle inequality, we get 
 $| S_{k\ell} - \Sh_{k\ell} |\le T_{21} + T_{22}$ where
\begin{align*}
    T_{21} := \sum_{i,j} A_{ij} |  1\{z_i = k, z_j = \ell\} - 1\{\zh_i = k, z_j = \ell\} |, \\
    T_{22} := \sum_{i,j} A_{ij} |  1\{\zh_i = k, \zh_j = \ell\} - 1\{\zh_i = k, z_j = \ell\} |. 
\end{align*}
Consider $T_{22}$ first. We have
\begin{align*}
    T_{22} &= \sum_j \Bigl( \sum_i A_{ij} 1\{\zh_i = k\} \cdot |  1\{\zh_j = \ell\} - 1 \{z_j = \ell\} |\Bigr) \\
    &\le \sum_j \Bigl( \sum_i A_{ij}  
    \cdot |  1\{\zh_j = \ell\} - 1 \{z_j = \ell\} |\Bigr) \\
    &\le 2C_1 \nu_n \cdot d_\ell(z,\zh)
\end{align*}
on $\Ec_0$. Similarly, $T_{21} \le 2 C_1 \nu_n \cdot d_k(z,\zh)$ on $\Ec_0$. Then,
\begin{align}\label{eq:Sh:S:dev}
    |S_{k\ell} - \Sh_{k\ell}| \le 2C_1 \nu_n \cdot(n_\ell + n_k) \rho \le 4 C_1 n \nu_n \rho 
\end{align}
using $n_k \le n$ for all $k$. Combined with~\eqref{eq:rho:nhamming}, this proves~\eqref{eq:S:Sh:max:dev}.

\medskip
Let us define $\Bt' = \Sh / m$. For $k \neq \ell$, we have $m_{k\ell} = n_k n_\ell$.
Then, on $\Ec_1$,
\begin{align*}
    | \Bt_{ k \ell} - \Bt_{k\ell}'| = \frac{|S_{k\ell} - \Sh_{k\ell}|}{n_k n_\ell} \le 
    2C_1 \nu_n \cdot(n_k^{-1} + n_\ell^{-1}) \rho  \le 4C_1 \beta \frac{\nu_n}{n} \rho
\end{align*}
using $n_k \ge n /\beta$ for all $k$. By~\eqref{eq:Bt:upper:bound}, on $\Ec_2$, we also have
\begin{align*}
    \Bt'_{k\ell} \le 2C_1 \frac{\nu_n}n (1 + 2 \beta \rho) \le 4C_1 \frac{\nu_n}n.
\end{align*}
assuming $2\beta \rho \le 1$.
Next we note that 
\[
\Bigl| \frac{\mh_{k\ell}}{m_{k\ell}} - 1   \Bigr| = 
\Bigl| \frac{\nh_k \nh_\ell}{n_k n_\ell} - 1 \Bigr|  \le 3\rho.
\]
Assuming that $3\rho \le 1/2$, letting $x = \mh_{k\ell}/ m_{k\ell}$, we have $|1-x| \le 3\rho \le 1/2$, hence $|1-x^{-1|}| \le 6\rho$. It follows that on $\Ec_2$
\[
  |\Bt_{k\ell}'- \Bh_{k \ell}| = \frac{\Sh_{k\ell}}{m_{k\ell}} \Bigl| 1 - \frac{m_{k\ell}}{\mh_{k\ell}} \Bigr| \le \Bt_{k\ell}' \cdot 6 \rho \le 24 C_1 \frac{\nu_n}n \rho.
\]

By triangle inequality 
\begin{align*}
	\maxnorm{\Bh - \Bt} &\le 
    \maxnorm{\Bh - \Bt'} + \maxnorm{\Bt' - \Bt} \\
    &\le 24 C_1 \frac{\nu_n}n \rho + 4C_1 \beta \frac{\nu_n}n \rho \\
    &\le 28 C_1 \frac{\nu_n}n \beta \rho
\end{align*}
using $\beta \ge 1$ (we are weakening the constants in favor of a simpler expression).

\bigskip

For~\eqref{eq:m:mh:max:dev}, we have $|[m - \mh]_{k\ell}|\le  n_k n_\ell |\frac{\nh_k \nh_\ell}{n_k n_\ell} - 1 | \le n^2 3 \rho$.
Putting the pieces together combined with~\eqref{eq:rho:nhamming} proves the claim.
\end{proof}

\subsection{Multi-network extension}
We write $(A,z) \sim \sbm_n(B, \pi)$ for an $n$-node draw from a Bayesian SBM, with connectivity $B$ and class prior $\pi$.
\begin{proposition}\label{prop:conn:consist:multi}
    Assume that $(A_t,z_t) \sim \sbm_{n_t}(B, \pi)$ for $t \in [N]$, where $B$ satisfies~\eqref{eq:sparsity} and $n_t$ satisfy
    \begin{align}
        n \le n_t \le C_2 n.
    \end{align}
    Let
    $
    S_t = \sums(A_t, z_t),\;  \Sh_t = \sums(A_t, \zh_t)$ and $
    m_t = \counts(z_t),\; \mh_t = \counts(\zh_t)$. For $\sigma = (\sigma_t) \in \Pi_K^{\otimes N}$, set
    \begin{align*}
        \Bh(\sigma) := \frac{
        \sum_t P_{\sigma_t} \Sh_t P_{\sigma_t}^T
        }{
        \sum_t P_{\sigma_t} \mh_t P_{\sigma_t}^T
        }, \quad \Bt := \frac{\sum_t S_t}{\sum_t m_t}, \quad 
        \eps_t(\sigma) &:= \nhamming(z_t, \sigma_t\circ \zh_t ) \end{align*}
    where we interpret the division of matrices as elementwise.

     Fix $\kappa \in (0,1)$ and $\alpha > 0$, let $\beta = 1/(\kb \pmin)$ with $\kb = 1-\kappa$, and define 
    \[
    \delta_n := \sqrt{\frac{3 \beta^2 \alpha \log n}{ N n \nu_n}}.
    \]
    Assume that $\delta_n \le 1$ and $ \nu_n \ge 3(1+\alpha) \log n$.
    Then, with probability at least 
    $1 - (C_2 N + 2K^2) n^{-\alpha} - NK e^{-\kappa^2 n \pmin/3}$, 
     we have, for all $\sigma = (\sigma_t) \in \Pi_K^{\otimes N}$ for which $\max_t \eps_t(\sigma) \le 1/(2\beta^3)$,
	\begin{align}
        \maxnorm{\Bh(\sigma) - \Bt} &\,\le\, 
         40 \,C_1 \beta^3 \frac{\nu_n}n \eps,
        \label{eq:Bh:Bt:max:dev:mult}\\
	\maxnorm{\Bt - B} &\,\leq\,  
        C_1 \frac{\nu_n} n \delta_n, 
        \label{eq:B:Bt:max:dev:mult}\\
        \maxnorm{S_t - P_{\sigma_t} \Sh_t P_{\sigma_t}^T} &\,\le\, 
        8 C_1 \beta \cdot \nu_n n_t \,\eps_t(\sigma),
        \label{eq:S:Sh:max:dev:mult} \\
        \maxnorm{m_t - P_{\sigma_t} \mh_t P_{\sigma_t}^T} &\,\le\, 
         6\beta \cdot n_t^2 \,\eps_t(\sigma), 
         \label{eq:m:mh:max:dev:mult}
        \\
        \min_{k, \ell}\, [m_t]_{k\ell} &\;\ge\; n_t^2 / \beta^2, \quad \text{for all $t \in [N]$}
	\end{align}
 where $\beta = 1/(\kb \pmin)$ with $\kb = 1-\kappa$. \end{proposition}
Note that $n_t$ here is deterministic (size of the $t$-th network) and different from $n_k$ in the proof of Proposition~\ref{prop:conn:consist}.

\begin{proof}
By redefining $\zh_t$ to be $\sigma_t(\zh_t)$, we can assume, without loss of generality, that $\sigma_t$ is the identity permutation (and hence $P_{\sigma_t} = I_K$) for all $t$. Note that none of the events we consider below depend on $\sigma$, hence the result indeed holds for all $\sigma$ simultaneously.

\medskip
First, consider event
\begin{align}
    \Ec_0 := \Bigl\{ \max_{t \in [N], \,j \in [n_t]}\sum_{i=1}^{n_t} [A_t]_{ij} \le 2C_1 C_2\nu_n \Bigr\}.
\end{align}
Applying Proposition~\ref{prop:binom:concent} conditioned on $z_t$, with $p = C_1 \nu_n /n$ and $n_t p \le C_1 C_2 \nu_n$, and then taking expectation of both sides to remove the conditioning, we have
\begin{align*}
    \pr\Bigl( \sum_{i=1}^{n_t} [A_t]_{ij}  \ge C_1 C_2 \nu_n (1+u) \Bigr) \le e^{-u^2 C_1 C_2 \nu_n/3}.
\end{align*}
Take $u = 1/\sqrt{C_1 C_2} \le 1$. Then, by union bound and assumption $\nu_n /3 \ge (1+\alpha) \log n$, we have $\pr(\Ec_0^c) \le C_2 N n^{-\alpha}$.

Next, let $n_{tk} := \sum_{i=1}^{n_t} 1\{(z_t)_i = k\}$, the (true) number of nodes in community~$k$ in network~$t$.  By Proposition~\ref{prop:binom:concent}, we have $\pr(n_{tk} \le \kb n_t \pi_k) \le \exp(-\kappa^2 n_t \pi_k/3)$ for $\kappa \in (0,1)$.
Consider the event defined as 
\[
   \mathcal{E}_1\,:=\,\Bigl\{  \underset{k =1,\ldots K}{\min}  \,n_{tk}\, \geq   n_t / \beta ,\; \forall t \in [N] \Bigr\},
\]
where $\beta = 1/(\kb \pmin)$ with $\kb = 1-\kappa$. Then, by union bound and using $n_t \ge n$, 
\[
\pr( \Ec_1^c ) \le N K \exp(-\kappa^2 n \pmin/3).
\]

By the same argument as in the proof of Proposition~\ref{prop:conn:consist},and letting $\eps_t = \nhamming(z_t, \zh_t)$, we have on $\Ec_0 \cap \Ec_1$,
\begin{align}\label{eq:b:d:control}
    \maxnorm{S_t - \Sh_t} \le 8 C_1 \beta \nu_n n_t \eps_t, \quad \maxnorm{m_t - \mh_t} \le 6 \beta n_t^2 \eps_t, \quad [m_t]_{k\ell} \ge n_t^2 / \beta^2.
\end{align}

Let us now control $\Bt = (\sum_t S_t) / (\sum_t m_t)$.
Conditioned on $z_t$, we have 
\[
[\textstyle\sum_t S_t]_{k\ell} \sim \bin\bigl([\textstyle x\sum_t m_t]_{k\ell}, B_{k\ell} \bigr).
\]
Also note that, for $k \neq \ell$, we have $[\sum_t m_t]_{k\ell} = \sum_t n_{tk} n_{t\ell} \ge \sum_t n_t^2 / \beta^2 \ge N n^2 / \beta^2$ on $\Ec_1$.
Then, applying Proposition~\ref{prop:binom:concent} conditioned on $z_t$, we get
\begin{align*}
	\pr\bigl(|B_{k\ell} -\Bt_{ k \ell} | \ge \delta_n B_{k\ell} \; \given z_t \bigr) \cdot 1_{\Ec_1} 
	&\;\le\; 2 e^{- \delta_n^2 [\sum_t m_t]_{k\ell} B_{k\ell} / 3} \cdot 1_{\Ec_1}\\
	&\;\le\; 2 e^{- \delta_n^2 (N n^2 / \beta^2)  (\nu_{n} / n)/ 3}
\end{align*}
where we have used the lower bound in~\eqref{eq:sparsity}. 
 Take $\delta_n^2 = 3 \beta^2 \alpha \log n / (N n \nu_n)$ and let
\begin{align}
	\Ec_2 = \Bigl\{ | B_{k\ell} - \Bt_{k\ell}| \le \delta_n 
 B_{k\ell} \;\; \text{for all}\; k,\ell \in [K] \Bigr\}.
\end{align}
Then, by union bound, $\pr(\Ec_1 \cap \Ec_2^c) \le \pr(\Ec_2^c \given \Ec_1) \le 2 K^2 n^{-\alpha}$. We will work on $\Ec_0 \cap \Ec_1 \cap \Ec_2$ and  note that 
\[
\pr(\Ec_0 \cap \Ec_1 \cap \Ec_2) \ge 1 - \pr(\Ec_0^c) - \pr(\Ec_1^c) - \pr(\Ec_1 \cap \Ec_2^c).
\]
Note that on $\Ec_2$,
 \begin{align}
     \max_{k,\ell}| B_{k\ell} - \Bt_{k\ell}| 
     &\le \delta_n \cdot C_1 \nu_n /n, 
     \label{eq:B:Bt}\\
     \max_{k,\ell} \Bt_{k\ell} &\le 2 C_1 \nu_n / n  
     \label{eq:Bt:upper:bound}
 \end{align}
 using the upper bound in~\eqref{eq:sparsity} and assumption $\delta_n \le 1$. This proves~\eqref{eq:B:Bt:max:dev}.

\medskip
Next we control the deviation $\Bh - \Bt$.
Treating division (and taking absolute values) of matrices as element-wise,
\[
\frac{\sum_t \Sh_t}{\sum_t \mh_t} = 
\frac{\sum_t S_t + \sum_t(\Sh_t - S_t)}{\sum_t m_t  +  \sum_t(\mh_t - m_t)} =: \frac{a + b}{c + d}.
\]
We then have
\[
\Bigl| \frac{a+b}{c+d} - \frac{a}{c} \Bigr| = \Bigl| \frac{(a/c)+(b/c)}{1+(d/c)} - (a/c) \Bigr| = \Bigl|\frac{(b/c) - (a/c) (d/c)}{1+ (d/c)} \Bigr| \le \frac{|(b/c) - (a/c) (d/c)|}{1 - |d/c|}.
\]
Let $\eps = \max_t \eps_t$.
Then, we have
\begin{align*}
    |[b/c]_{k\ell}| 
    \le \frac{\sum_t |[\Sh_t - S_t]_{k\ell}|}{\sum_t [m_t]_{k\ell}} 
    &\stackrel{(a)}{\le} 8C_1 \beta^3 \nu_n \frac{\sum_t  n_t \eps_t}{\sum_t n_t^2} \\
    &\stackrel{(b)}{\le} 8C_1 \beta^3 \frac{\nu_n}n \frac{\sum_t  n_t \eps_t}{\sum_t n_t} 
   \le 8C_1 \beta^3 \frac{\nu_n}n  \eps
\end{align*}
where (a) is by~\eqref{eq:b:d:control} and (b) uses assumption $n_t \ge n$.
Similarly, we have $|[d/c]_{k\ell}| \le 6 \beta^3 \eps$. Note that $a/c = \Bt$ hence elementwise in $[0, 2C_1 \nu_n /n]$ on $\Ec_2$. Assuming that $\beta^3 \eps \le 1/2$, we have,
\[
\frac{|(b/c) - (a/c) (d/c)|}{1 - |d/c|} \le
\frac{|b/c| + (a/c) |d/c|}{1 - \frac12} \le 40 \,C_1 \beta^3 \frac{\nu_n}n \eps.
\]
(where the final inequality means every element of the LHS is $\le$ RHS) .
That is, we have shown
\[
\Bigl| \frac{\sum_t \Sh_t}{\sum_t \mh_t} - \frac{\sum_t S_t}{\sum_t m_t}\Bigr| \le 40 \,C_1 \beta^3 \frac{\nu_n}n \eps.
\]
This proves~\eqref{eq:Bh:Bt:max:dev:mult} and finishes the proof. \end{proof}
 
\section{Remaining Proofs}
\subsection{Proof of Proposition~\ref{prop:matching:recovery}}

Recall that $\infnorm{\Delta} = \max_{i,j} | \Delta_{ij}|$. We need the following lemma on the perturbation of the LAP problem:

\begin{lemma}\label{lem:lap:pertub:at:identity}
	We have  $\lap(I + \Delta) = \{I\}$ as long as $\infnorm{\Delta} < 1/2$.
\end{lemma}

Since $B_1$ and $B_2$ are similar matrices, they have the same eigenvalues. If the EVD of $B_1$ is given by $B_1 = Q_1 \Lambda Q_1^{\top}$, then the EVD of $B_2$ can written as $B_2 = Q_2 \Lambda Q_2^{\top}$. 
By Lemma~\ref{lemma:q2 = pq1s},
\begin{equation}\label{eq:Q2:Q1}
    Q_2 = P^* Q_1 S^*
\end{equation}
for some sign matrix $S^*$. 

First, we show assertion~(a). Let $\Delh_r := P_r \Qh_r S_r - Q_r$, so that $\Qh_r = P_r^{\top} (Q_r+\Delh_r) S_r$ and 
\begin{align*}
	\Qh_r^{\top} \allone = S_r ( Q_r+\Delh_r)^{\top} \allone
\end{align*}
using $P_r \allone = \allone$. Then, $	[\Qh_r^{\top} \allone]_k = S_{r,kk} (Q_{r, * k}^{\top} \allone + \Delh_{r,*k}^{\top}  \allone)$
where $Q_{r,*k}$ and $\Delta_{r,*k}$ are the $k$-th columns of $Q_r$ and $\Delta_r$. From the definition of $\hat{S}_{kk}$,
\begin{align}
	\Sh_{kk} & = \sign \left( \frac{[\Qh_1^{\top}  \allone]_k}{[\Qh_2^{\top}  \allone]_k}  \right) =
	\sign \left(\frac{Q_{1,*k}^{\top}  \allone + \Delh_{1,*k}^{\top}  \allone}{ Q_{2,*k}^{\top}  \allone + \Delh_{2,*k}^{\top}  \allone} \right) S_{1,kk} S_{2,kk}. 
\end{align}
From~\eqref{eq:Q2:Q1}, it follows that $S^*_{kk} = \sign \bigl( [Q_2^{\top} \allone]_k / [Q_1^{\top} \allone]_k \bigr)$.
Then, to show
 $\Sh_{kk} = S_{1,kk} S_{kk}^* S_{2,kk}$, it suffices to show that
\begin{equation}
\begin{aligned}
        |Q_{r,*k}^{\top}  \allone| = \Bigl|\sum_{j=1}^K Q_{r, jk} \Bigr| & >  \Bigl|\sum_{j=1}^K \Delh_{r, jk} \Bigr| = |\Delh_{r,*k}^{\top} \allone|
\end{aligned}
\end{equation}
for $r = 1,2$. Since $B_r$ are $(\theta, \eta)$-friendly, we have $|\sum_{j=1}^K Q_{r, jk}|  \geq \theta$
and so it is enough to show that
$
 |\sum_{j = 1}^K \Delh_{r, jk} | \leq \norm{\Delh_{r, * k}}_1  < \theta.
$

The noise matrices $\Delh_r$ are controlled by the Davis--Kahan theorem, 
\begin{align*}
    \norm{\Delh_{r, *k}}_2 & = \norm{Q_{r, *k} - P_r\Qh_{r, *k} S_{r, kk} }_2 \\
    & \leq \frac{2\sqrt{2} }{\eta} \fnorm{B_r - P_r\Bh_rP_r^{\top}}.
\end{align*}
Then, using assumption~\eqref{eq:eps:delta:assump},
\begin{equation}
    \norm{\Delh_{r, *k}}_1 \leq \sqrt{K} \norm{\Delh_{r, *k}}_2 \leq \frac{2\sqrt{2K}}{\eta} \frac{\theta \eta}{2\sqrt{2}K} \le \theta,
\end{equation}
proving assertion~(a). Note also that we have shown
\begin{align}\label{eq:ell1:mnorm:Delh}
	\mnorm{\Delh_r}_1 = \max_{k} \norm{\Delh_{r, *k}}_1  \le \theta.
\end{align}

Next, we prove $\lap(\Qh_1 \St \Qh_2^{\top}) = \{\Pt\}$. Using~\eqref{eq:Q2:Q1}, and the definitions of $\Delh_r$, we have
\begin{align*}
\Qh_2 & = P_2^{\top}(Q_2+\Delh_2)S_2\\
& = P_2^{\top}(P^* Q_1 S^* + \Delh_2)S_2 \\
& = P_2^{\top}(P^*(P_1 \Qh_1 S_1 - \Delh_1)S^*+ \Delh_2)S_2 \\
& = P_2^{\top}P^*P_1 \Qh_1 S_1 S^* S_2 - P_2^{\top} P^*\Delh_1 S^* S_2 + P_2^{\top} \Delh_2 S_2\\
& = \Pt \Qh_1 \St + \Delta
\end{align*}
where we let $\Delta = -P_2^{\top} P^*\Delh_1 S^* S_2 + P_2^{\top} \Delh_2 S_2$. We can then write
\begin{equation}
    \Qh_1 \St \Qh_2^{\top} = \Qh_1 \St (\Pt \Qh_1 \St + \Delta)^{\top} =
    \Pt^{\top} (I + \Delta_0).
\end{equation}
where $\Delta_0 =  \Pt \Qh_1 \St \Delta^{\top}$. It is then enough to study \begin{align*}
    \lap\bigl(\Pt^{\top}(I + \Delta_0) \bigr) =  \lap(I + \Delta_0) \cdot \Pt
\end{align*}
where the equality follows by a change-of-variable argument. The result follows from Lemma~\ref{lem:lap:pertub:at:identity} if we show $\infnorm{\Delta_0} \le 1/2$. 

We note that for permutation and sign matrices, both $\mnorm{\cdot}_\infty$ and $\mnorm{\cdot}_1$ are equal to 1. Using the submultiplicative property of $\mnorm{\cdot}_p$ for $p =1,2$, we have
\begin{align*}
	 \mnorm{\Delta^{\top}}_{\infty}  = \mnorm{\Delta}_1 \le \mnorm{\Delh_1}_1 + \mnorm{\Delh_2}_1 < 2\theta
\end{align*}
where we have used~\eqref{eq:ell1:mnorm:Delh}. Then, 
\begin{align*}
     \mnorm{\Delta_0}_{\infty} & \leq \mnorm{\Qh_1}_\infty \mnorm{\Delta^{\top}}_\infty \le 2 \theta  \mnorm{\Qh_1}_\infty.
\end{align*}
Since $\Qh_1$ has unit-norm rows, that is, $\norm{\Qh_{1,k*}}_2 = 1$ for all $k$, we obtain
\begin{align*}
	\mnorm{\Qh_1}_\infty = \max_{k} \norm{\Qh_{1,k*}}_1 \le \sqrt{K}. \end{align*}
Putting the pieces together
\begin{align*}
	\infnorm{\Delta_0} \le \mnorm{\Delta_0}_\infty \le 2 \theta \sqrt{K} < 1/2
\end{align*}
where the last inequality is by assumption. This proves $\lap(\Qh_1 \St \Qh_2^{\top}) = \{\Pt\}$.

\medskip
To prove the inequality in part (b), let $D_r := B_r - P_r \Bh_r P_r^{\top}$. Then, some algebra, using $B_2 = P^* B_1 P^{*T}$, gives 
\begin{align*}
\Pt \Bh_1 \Pt^{\top} 
& = P_2^{\top} P^{*} P_1 \Bh_1 P_1^{\top} P^{*T} P_2 \\
& = P_2^{\top} P^*(B_1 - D_1)P^{*T} P_2 \\
& = P_2 P^* B_1 P^{*T} P_2 - P_2^{\top} P^*D_1 P^{*T} P_2 \\
& = P_2^{\top} B_2 P_2 - P_2^{\top} P^* D_1 P^{*T}P_2 \\
&= P_2^{\top} D_2 P_2 + \Bh_2 -  P_2^{\top} P^* D_1 P^{*T}P_2,
\end{align*}
and so 
\begin{align*}
    \fnorm{\Pt \Bh_1 \Pt^{\top} - \Bh_2} & \leq \fnorm{P_2^{\top} D_2 P_2 - P_2^{\top} P^* D_1 P^{*T} P_2} \\
    &\le \fnorm{D_1} + \fnorm{D_2}.
\end{align*}
The proof is complete.
 
\section{Technical Lemmas}\label{app:tech:lemmas}

We have the following result, which follows for example from Bernstein's inequality:
\begin{proposition}\label{prop:binom:concent}
 Suppose that $S = \sum_{i=1}^n X_i$ where $X_i \sim \bern(p_i)$ independently for $i \in [n]$. Assume that $\max_i p_i \le p$.
 Then, for all $\delta \in [0,1]$,
	\[
	\pr\Bigl(\ph - p \ge \delta p \Bigr) \le e^{-\delta^2 np /3},
	\]
	and the same inequality holds for $\pr(p - \ph \ge \delta p)$.
\end{proposition}
\begin{proof}

By Bernstein inequality, using $|X_i -p_i| \le 1$, for any $u>0$, we have 
    \begin{align*}
        \pr( S - \ex[S] \ge n u) \le \exp \Bigl( - \frac{n u^2}{ 2 \bar \sigma^2  + 2 u/3}\Bigr) 
    \end{align*}
    where $\bar \sigma^2 = \frac1n \sum_{i=1}^n \var(X_i)$. We have $\bar \sigma^2 \le p$ and  $\ex[S] \le np$. Setting $u = \delta p$, we have 
    \begin{align*}
        \pr( S - np \ge \delta n p) &\le \exp \Bigl( - \frac{n \delta^2 p^2}{ 2 p  + 2 \delta p/3}\Bigr)  \\
        &= \exp \Bigl( - \frac{ \delta^2 }{ 2   + 2 \delta /3} n p\Bigr) \le \exp(- 3 \delta^2 n p / 8)
    \end{align*}
    where the last inequality uses $\delta \le 1$. Replacing $3/8$ with $1/3$ gives a further upper bound.
\end{proof}

\section{Proof of Auxiliary lemmas}

\subsection{Proof of Lemma~\ref{lemma:q2 = pq1s}}
	Since $P^* Q_1$ is an orthogonal matrix, by absorbing $P^*$ into $Q_1$ and redefining $Q_1$, we can assume $P^* = I$, without loss of generality. The problem reduces to showing that 
	(*) $Q_2 \Lambda Q_2^T = Q_1 \Lambda Q_1^T$
	iff there is a sign matrix $S^*$ such that $Q_2 = Q_1S^*$. Let $Q = Q_2^T Q_1$ and note that $Q$ is an orthogonal matrix. Multiplying~(*) on the left and right by $Q_2^T$ and $Q_2$, the problem reduces to showing that (**) $\Lambda = Q \Lambda Q^T$ for an orthogonal matrix $Q$ and diagonal matrix $\Lambda$ iff $Q$ is a diagonal sign matrix (i.e., $Q = S^*$).
	
	Assume (**) holds, the other direction being trivial. Note that changing $\Lambda$ to $\Lambda + \alpha I$ does not change (**), hence we can shift the diagonal entries of $\Lambda$ arbitrarily. Let $\Lambda = \diag(\lambda_k)$.  Since $(\lambda_k)$ are distinct, we can shift them so that $\lambda_1 < 0$ and $\lambda_k >0$ for $k \ge 2$. From (**), looking at the first entries of the two sides, $\lambda_1 = \sum_{k\ge 1} Q_{ik}^2 \lambda_k$, hence
	\[
	0 = -(1-Q_{11}^2)\lambda_1 + \sum_{k \ge 2} Q_{1k}^2 \lambda_k.
	\]
	Every term on the RHS is non-negative. It follows that every term has to be zero, implying $Q_{11}^2 =1$ and $Q_{1k} = 0$ for $k \ge 2$. This proves the assertion for the first row of $Q$. Repeating the argument for the other rows, the result follows.
	
	Next, we prove the uniqueness. Suppose that there exist  permutation matrices $P$ and $\tilde{P}$ such that $ PB_1 P^{T} = B_2 =  \tilde{P}B_1 \tilde{P}^{T}$.
	By the argument above, there exist sign matrices  $S$ and $\tilde{S}$ such that
	$P Q_1 S = Q_2, = \tilde{P} Q_1 \tilde{S}$.
	Hence, 
	\[
	Q_1  = P^{T} \tilde{P} Q_1 \tilde{S} S.
	\]
	The problem then reduces to showing that if $Q = P Q S$ where $Q$ is an orthogonal matrix, $P$ a permutation matrix and $S$ a sign matrix, then $P = I$. If $S = I$ (the trivial sign matrix), then $P = QQ^T = I$. So assume $S \neq I$ and $P \neq I$. Then, there is $j$ such that $P Q_{.,j} = - Q_{.,j}$, hence,
	\[
	\allone^T Q_{.,j} = \allone^T P Q_{.,j} = - \allone^T Q_{.,j}
	\]
	where $\allone$ is the all-ones vector.
	This gives $\allone^T  Q_{.,j} = 0$, contradicting friendliness of $B_1$. The proof is complete.
 
\subsection{Proof of Lemma~\ref{lem:lap:pertub:at:identity}}
	For $P \in \Pi_K$, let  
	\begin{align*}
	J_1(P)   \,&=\,\{ i \,:\,  P_{ii} \neq 0  \}, \\
	J_2(P) \,&=\,\{ (i,j) \,:\,  P_{ij} \neq 0,  \,\, i \neq j   \},
	\end{align*}
	and note that $|J_1(P)|+ |J_2(P)| = K$ for any $P \in \Pi_K$. By assumption $\infnorm{\Delta} < 1/2$,  and hence
	\begin{align*}
			1 +  \Delta_{ii}   >   1/2 &\quad \text{for}\quad  i \in J_1(P)\\
				\Delta_{ij} <   1/2 &\quad \text{for}\quad  (i,j) \in J_2(P).
	\end{align*}
	We also have
	\begin{align*}
			\tr\bigl(P^T (I+\Delta)\bigr) &= \sum_{i,j}   ( 1_{ \{   i=j \} } +   \Delta_{ij} ) P_{ij} \\
			&= \sum_{i \in J_1(P)} ( 1+   \Delta_{ii} ) \,+\, \sum_{(i,j) \in J_2(P)}  \Delta_{ij}.
	\end{align*}
	Let $P \neq I$, so that both $J_2(P)$ and $[K] \setminus J_1(P)$ are nonempty. Then,
	\begin{align*}
		\tr(I + \Delta) - \tr\bigl(P^T (I+\Delta)\bigr)  
		&= \sum_{i \in [K] \setminus J_1(P)} ( 1+   \Delta_{ii} ) \,-\, \sum_{(i,j) \in J_2(P)}  \Delta_{ij} \\
		&> \frac12 (K - |J_1(P)|) - \frac12 |J_2(P)| = 0,
	\end{align*}
	showing that identity is the unique optimal solution.

\section{Randomization}\label{sec:randomization}

Let  $\mathcal{U}(\Pi_K)$ denote the  uniform distribution on the set of permutation matrices $\Pi_K$.

\begin{lem}[Randomization]
    \label{lem:randomization}
    Assume that there is a permutation matrix $C_{rt} = C_{rt}(A_{rt})$ potentially dependent on the adjacency matrix $A_{rt}$ such that
\[
    \apermto{\Bt_{rt}, C_{rt}, B_{rt}}
\]
where      $\Bt_{rt}= \sums(A_{rt}, \zh_{rt}^{(0)})\oslash \counts(\zh_{rt}^{(0)})$. Let $\Bh_{rt}$  be constructed as in step~step~\ref{step:bhat_rt}.
Then, 
\begin{align}
\label{eqn:conclusion}
    \apermto{\Bh_{rt}, P_{rt}, B_{rt}}
\end{align}
where $P_{rt} \sim \mathcal{U}(\Pi_K)$ independent of $A_{rt}$ (and hence $\Bh_{rt}$). 
\end{lem}
\begin{proof}
Notice that by construction $\Bh_{rt} =  U_{rt} \Bt_{rt} U_{rt}^\top$ with $U_{rt}\sim \mathcal{U}(\Pi_K)$.    Let $P_{rt}  = C_{rt}U_{rt}^{\top} \in  \Pi_K$. Then,  
    $C_{rt} \,\Bt_{rt}\, C_{rt}^\top = P_{rt} \,\Bh_{rt}\, P_{rt}^\top$ and (\ref{eqn:conclusion}) follows. It remains to show independence of $P_{rt}$ and its uniform distribution. Indeed, let $P_0 \in \Pi_K$ and $A_0 \in \{0,1\}^{n\times n}$. We have (showing the dependence of $C_{rt}$ on $A_{rt}$ explicitly)
    \begin{align*}
          \pr(  P_{rt} =P_0 , A_{rt} = A_0 )  
          &= \pr(  U_{rt} = P_0^\top C_{rt}(A_{rt}), A_{rt} = A_0 )\\
    &= \pr( U_{rt} = P_0^\top C_{rt}(A_0), A_{rt} = A_0 )\\
    &=\pr( U_{rt} = P_0^\top C_{rt}(A_0)) \cdot \pr( A_{rt} = A_0 )\\
    &= \frac{1}{|\Pi_K|} \cdot \pr( A_{rt} = A_0 ) 
    \end{align*}
    where the third line is by the independence of $U_{rt}$ and $A_{rt}$ and the fourth line since $U_{rt} \sim \mathcal U(\Pi_K)$. The above shows that the joint distribution factorizes as uniform distribution for $P_{rt}$ and original (marginal) distribution for $A_{rt}$, which proves the claim.
\end{proof}

\section{Additional plots and simulations}

\subsection{ROC plots for general SBM with random $B$}
\label{app:mean:ROCs}
Figure~\ref{fig:random_B} shows the mean ROC curves for the experiment in Section~\ref{sec:exp-random-B}. The different plots correspond to different values of $K$. We refer to Section~\ref{sec:exp-random-B} for the detailed description of the experiments, where a summary of these curves via their ``area under the curve (AUC)'' was provided in Table~\ref{tab:general_sbm}.

\begin{figure}[H]
     \centering
     \begin{subfigure}[b]{0.45\textwidth}
         \centering
         \includegraphics[width=\textwidth]{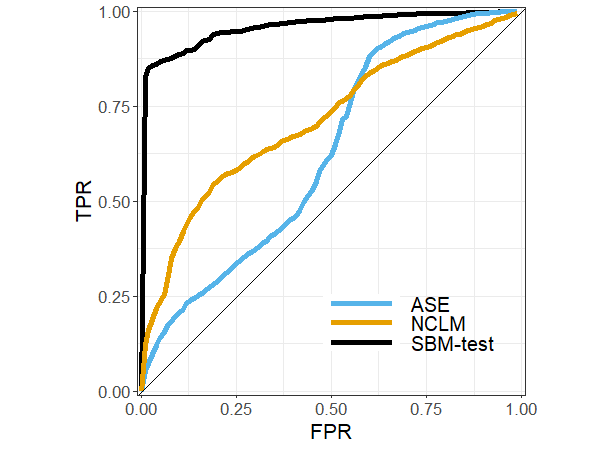}
         \caption{$K = 2$, mean AUC = 0.95.}
     \end{subfigure}
     \hfill
     \begin{subfigure}[b]{0.45\textwidth}
         \centering
         \includegraphics[width=\textwidth]{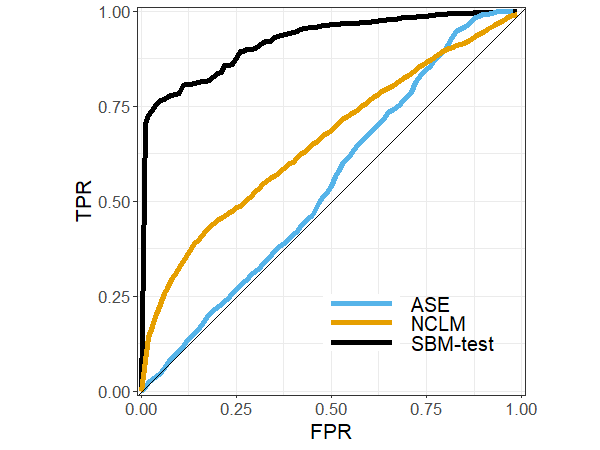}
         \caption{$K = 3$, mean AUC = 0.91.}
     \end{subfigure}
     \vskip\baselineskip
     \begin{subfigure}[b]{0.45\textwidth}
         \centering
         \includegraphics[width=\textwidth]{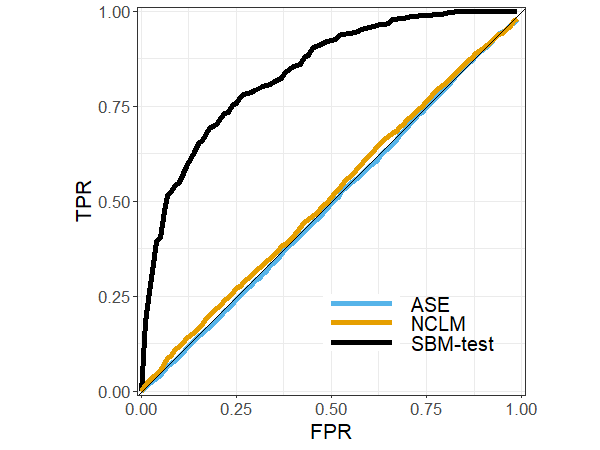}
         \caption{$K = 15$, mean AUC $ = 0.83$.}
     \end{subfigure}
     \hfill
     \begin{subfigure}[b]{0.45\textwidth}
         \centering
         \includegraphics[width=\textwidth]{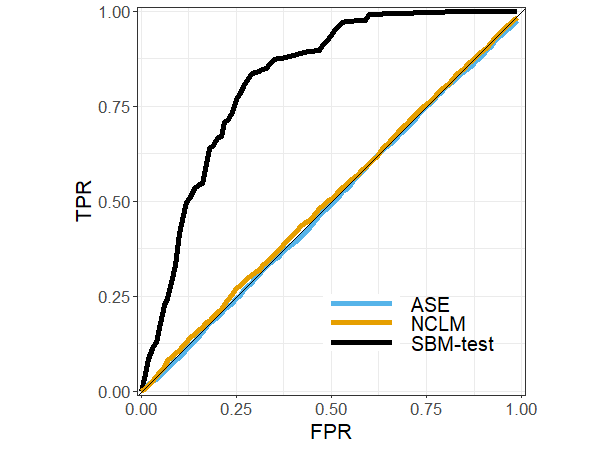}
         \caption{$K = 20$, mean AUC $= 0.81$.}
     \end{subfigure}
     \caption{ROC curves for the three methods (SBM-TS, MMD of ASE, and test statistic based on NCLM) averaged over 50 different experiments.}
     \label{fig:random_B}
\end{figure}

\subsection{Bandwidth of ASE-MMD}
\label{sec:appendix-bw}
We have conducted additional experiments in the same setting of Section~\ref{sec:exp-random-B}, that is, general SBM with random $B$ to determine the effect of bandwidth on the performance of ASE-MMD. The results are summarized in Figure~\ref{fig:random_B_bandwidth}. One observes that the bandwidth does not have a significant bearing on the  power of the ASE-MMD test, with ROCs remaining almost the same across the range of $\sigma^2 \in \{0.01, 0.1, 1, 10, 100\}$.

\begin{figure}[H]
     \centering
     \begin{subfigure}[b]{0.45\textwidth}
         \centering
         \includegraphics[width=\textwidth]{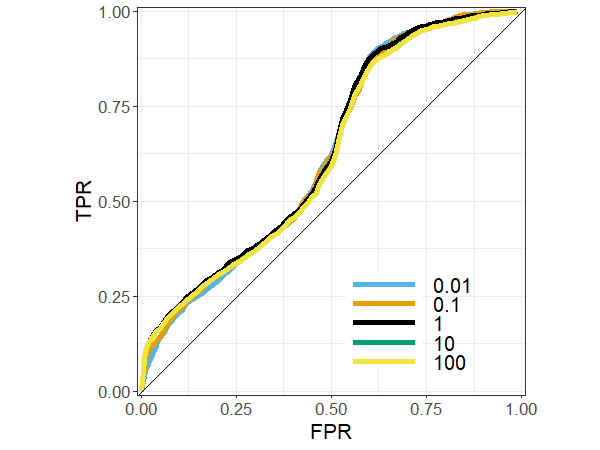}
         \caption{$K = 2$}
     \end{subfigure}
     \hfill
     \begin{subfigure}[b]{0.45\textwidth}
         \centering
         \includegraphics[width=\textwidth]{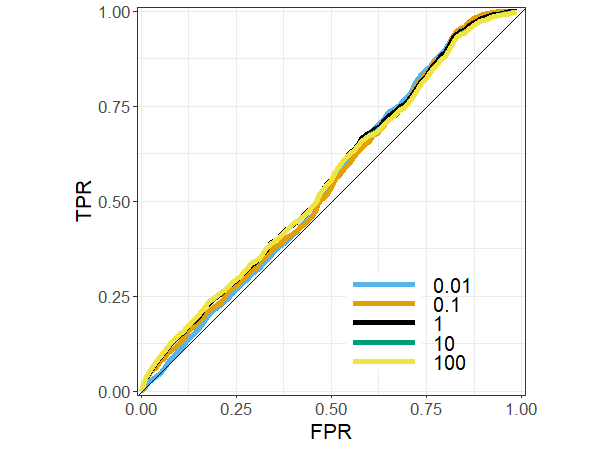}
         \caption{$K = 3$}
     \end{subfigure}
     \vskip\baselineskip
     \begin{subfigure}[b]{0.45\textwidth}
         \centering
         \includegraphics[width=\textwidth]{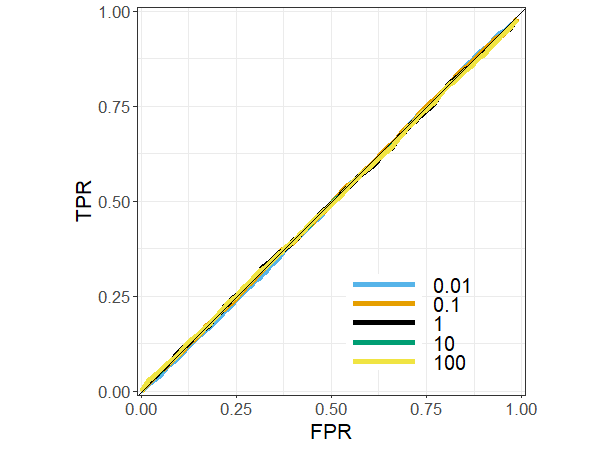}
         \caption{$K = 15$}
     \end{subfigure}
     \hfill
     \begin{subfigure}[b]{0.45\textwidth}
         \centering
         \includegraphics[width=\textwidth]{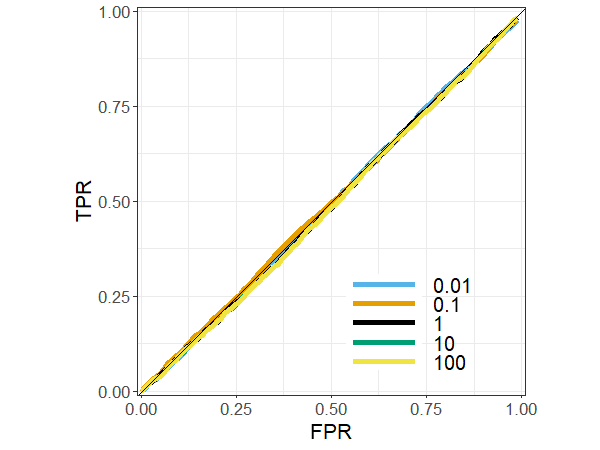}
         \caption{$K = 20$}
     \end{subfigure}
     \caption{ROC curves for different choices of the bandwidth $\sigma^2$ in the experiment with a general $B$.}
     \label{fig:random_B_bandwidth}
\end{figure}

\subsection{Number of Log Moments for NCLM}
\label{app:choice:of:J}
We have performed additional experiments in the same setting of Section~\ref{sec:exp-random-B}, that is, general SBM with random $B$ to determine the effect of the number $J$ of log-moments on the performance of NCLM. The results are summarized in Figure~\ref{fig:random_B_logmoments}. The general trend is that higher $J$ improves performance with $J = 20$ (the maximum we considered) producing the best results. We also note that this is mainly for small values of $K$, while for larger $K \in \{15,20\}$ the test is powerless regradless of the value of $J$.

\begin{figure}[H]
     \centering
     \begin{subfigure}[b]{0.45\textwidth}
         \centering
         \includegraphics[width=\textwidth]{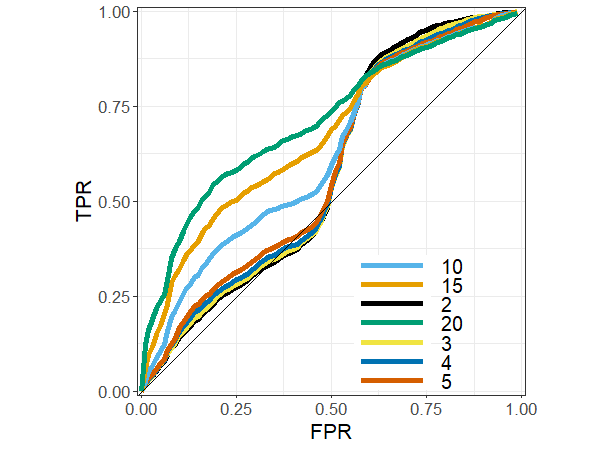}
         \caption{$K = 2$}
     \end{subfigure}
     \hfill
     \begin{subfigure}[b]{0.45\textwidth}
         \centering
         \includegraphics[width=\textwidth]{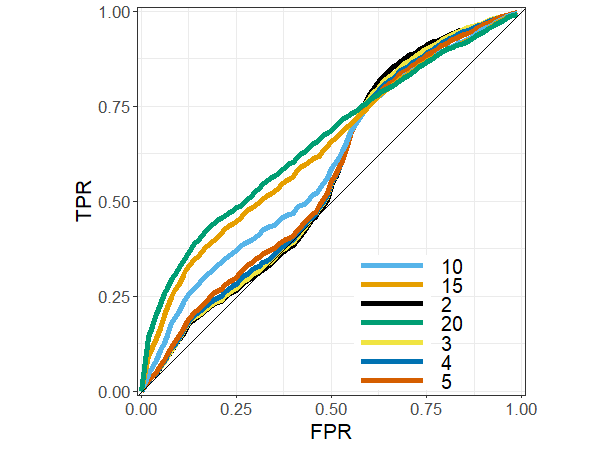}
         \caption{$K = 3$}
     \end{subfigure}
     \vskip\baselineskip
     \begin{subfigure}[b]{0.45\textwidth}
         \centering
         \includegraphics[width=\textwidth]{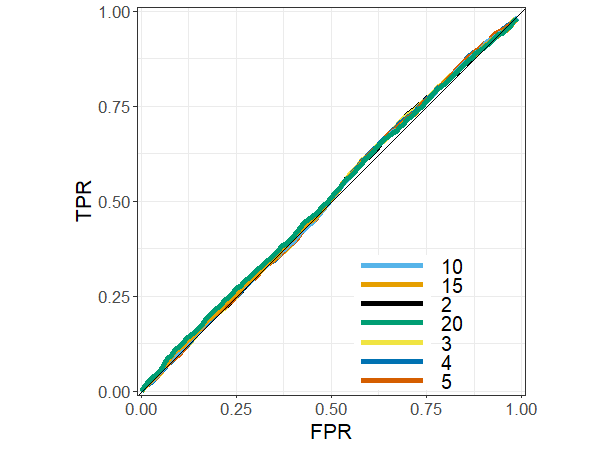}
         \caption{$K = 15$}
     \end{subfigure}
     \hfill
     \begin{subfigure}[b]{0.45\textwidth}
         \centering
         \includegraphics[width=\textwidth]{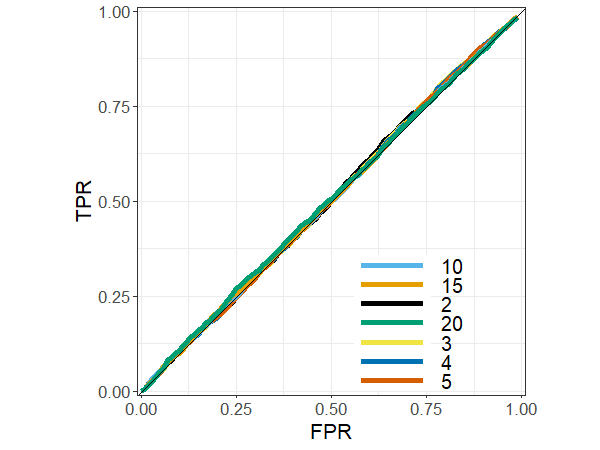}
         \caption{$K = 20$}
     \end{subfigure}
     \caption{ROC curves for different choices of the number of log-moments in the experiment with a general $B$.}
     \label{fig:random_B_logmoments}
\end{figure}

\subsection{Sample graphs for real-world data}
\label{app:sample:graphs}

\begin{figure}[H]
     \centering
     \includegraphics[width=.32\textwidth]{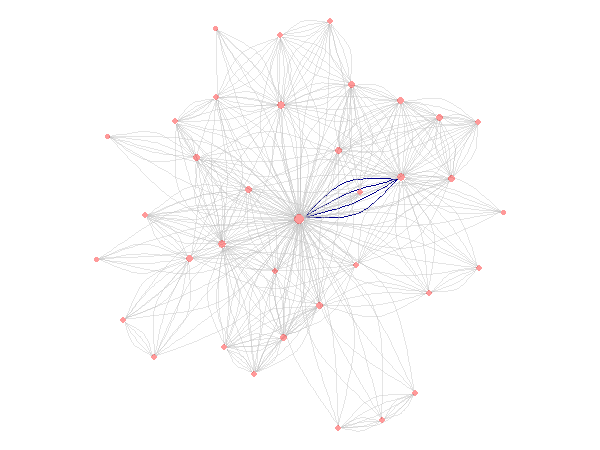}
     \includegraphics[width=.32\textwidth]{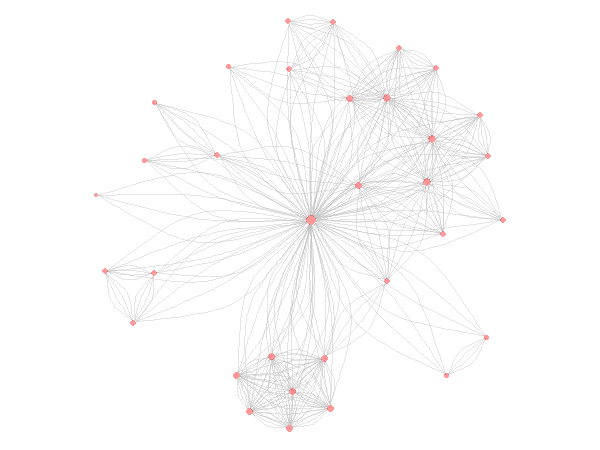}
     \includegraphics[width=.32\textwidth]{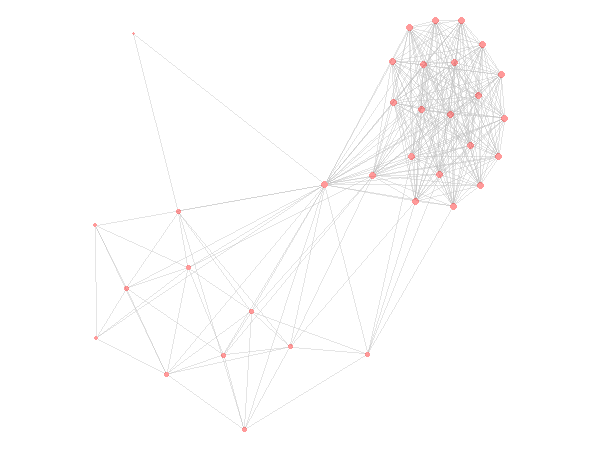}
    \caption{Sample graphs from the COLLAB dataset: High Energy Physics (left), Condensed Matter Physics (middle) and Astrophysics (right).}
    \label{fig:collab:graphs}
\end{figure}

\begin{figure}[H]
     \centering
     \includegraphics[width=.42\textwidth]{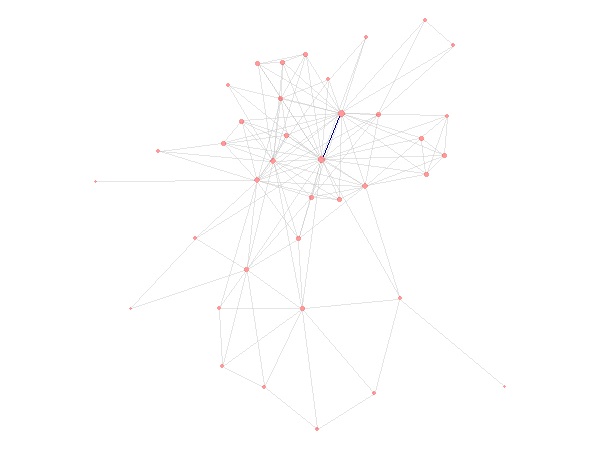}
     \includegraphics[width=.42\textwidth]{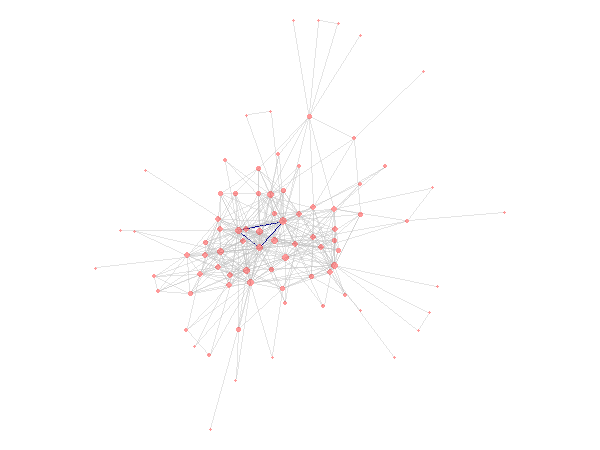}
     \caption{Sample graphs from the movie/television dataset: Class 1 (left) and Class 2 (right)}
     \label{fig:swgot_graphs}
\end{figure}

\end{document}